\documentclass{amsart}
\usepackage{amsfonts}
\usepackage{amsmath}
\usepackage{esint}
\usepackage{amssymb}
\usepackage{amsthm}
\usepackage{indentfirst}
\usepackage{xcolor}
\usepackage{enumitem}

\newtheorem{lemma}{Lemma}
\newtheorem{theorem}{Theorem}
\newtheorem{proposition}{Proposition}

\newtheorem{corollary}{Corollary}

\theoremstyle{remark}

\newcommand{\dv}{{\rm div}}

\newcommand{\li}[1]{\|#1\|_{L^\infty}}

\newcommand{\tor}{\mathbb T^d}

\begin{document}

\title{Inverse energy cascade and non-uniqueness for the advection-diffusion equation}

\author{Alexey Cheskidov}
\address{Institute for Theoretical Sciences, Westlake University, 600 Dunyu Road, Hangzhou, Zhejiang 310024, People's Republic of China}
\email{cheskidov@westlake.edu.cn}

\author{Xiaoran Liu}
\address{Institute for Theoretical Sciences, Westlake University, 600 Dunyu Road, Hangzhou, Zhejiang 310024, People's Republic of China}
\email{liuxiaoran@westlake.edu.cn}

\begin{abstract}
We prove non-uniqueness for the advection-diffusion equation with a drift
in a logarithmically near-critical class below
$L_t^2L_x^\infty$. The first construction is a diffusion-assisted inverse
cascade of scalar energy that diverges as $t \to 0$. The second uses
compressed inverse mixing, with diffusion acting perturbatively, and produces a bounded parabolic solution with the energy jump at the initial time.
In both constructions the drift and scalar are smooth for positive time. We also consider the drift in $L^{2,\infty}_tL^\infty_x$ and solutions satisfying the energy inequality starting from almost every time. In this setting, we prove uniqueness of solutions with bounded energy, and construct non-unique solution with energy blowing up as time goes to zero.
\end{abstract}

\maketitle

\section{Introduction}

We study the advection--diffusion equation on the flat torus $\tor$, $d\ge2$,
\begin{equation}\label{adv}
\begin{cases}
\partial_t\rho+b\cdot\nabla\rho-\Delta\rho=0,\\
\dv b=0,\\
\rho|_{t=0}=0,
\end{cases}
\end{equation}
where $\rho$ is a passive scalar and $b$ is a prescribed incompressible drift
\[
b \in C^\infty((0,\infty) \times \mathbb{T}^d).
\]
A function $\rho$ is a distributional solution of \eqref{adv} with zero initial datum if $\rho,b\rho\in L^1_{\rm loc}$ and
\begin{equation}\label{weak-form-intro}
\int_0^\infty\!\int_{\tor}
\rho\bigl(\partial_t\varphi+b\cdot\nabla\varphi+\Delta\varphi\bigr)
\,dx\,dt=0
\end{equation}
for every $\varphi\in C_c^\infty([0,\infty)\times\tor)$.
Since $b$ is smooth for positive time, every distributional solution 
\[
\rho \in C^\infty((0,\infty) \times \mathbb{T}^d),
\]
by interior parabolic regularity. A distributional solution in $L^\infty(0,T;L^2)\cap L^2(0,T;H^1)$ is called a parabolic solution.  

We study how the behavior of the drift as $t \to 0^+$ affects uniqueness, i.e., whether there are nonzero solutions to \eqref{adv}.

Due to positive time smoothness, $\rho$ satisfies the energy equality for every $0<t_0<t$:
\begin{equation}\label{energy-intro}
\frac12\|\rho(t)\|_{L^2}^2 +\int_{t_0}^t\|\nabla\rho(\tau)\|_{L^2}^2\,d\tau
=\frac12\|\rho(t_0)\|_{L^2}^2.
\end{equation}
Hence $\rho \in L^\infty(0,T;L^2)$ if and only if $\rho \in L^2(0,T;H^1)$.
A solution $\rho \equiv 0$ is zero if an only if the energy equality is satisfied starting from $t_0=0$.

We study the following question:
\begin{center}
{\it In what classes the energy equality starting from $t_0=0$, and hence uniqueness holds?}
\end{center}
 
The answer depends on the behavior of the drift as $t\to 0^+$. We show that if $b \in L_t^{2,\infty}L_x^\infty$, then nonzero $\rho$ cannot have finite energy jumps at the initial time. 

Previously, Bonicatto, Ciampa, and Crippa \cite{BonicattoCiampaCrippa2024} proved uniqueness of distributional solutions in the class of bounded energy $\rho \in L^\infty_t L^2$ provided $b\in L_t^2L_x^\infty$. When $b$ is smooth for positive time, we show that the uniqueness of bounded energy solutions holds even when $b\in L_t^{2,\infty}L_x^\infty$.  In fact, smoothness of $b$ can be relaxed to the energy inequality starting from almost every time, see Theorem~\ref{endpoint-main}. We then present three different constructions showing the sharpness of this condition.

\vspace{1em}

We now state our main results. For $A>0$, define the function

{$$
G_A=s^2\exp\!\left[-A\bigl(\log(e+s)\bigr)^{3/4}\right]
$$}

Our first two results demonstrates non-uniqueness under a logarithmically near-critical integrability condition. The first one is diffusion-assisted, and the second one is with perturbative diffusion.

\begin{theorem}[Diffusion-assisted non-uniqueness via instantaneous blow-up] \label{main}
There exists a divergence-free drift
\[
b\in C^\infty((0,\infty)\times\tor), \qquad
\int_0^\infty G_A\!\left(\|b(t)\|_{L_x^\infty}\right)\,dt < \infty, \quad \text{for every }A>0,
\]
and a nonzero function
\[
\rho\in C^\infty((0,\infty)\times\tor)
\cap\bigcap_{1\le p<\infty}L^p((0,\infty);L^\infty(\tor)),
\]
with
\[
\lim_{t \to 0+} \|\rho(t)\|_{L^2} = \infty,
\]
which solves \eqref{adv} with zero initial datum in the sense of \eqref{weak-form-intro}.
Hence, the zero solution is not unique within this class of distributional solutions.
\end{theorem}

\begin{theorem}[Inverse-mixing non-uniqueness with perturbative diffusion]\label{main2}
There exists a divergence-free drift
\[
b\in C^\infty((0,\infty)\times\tor),
\]
such that

{
$$
\int_0^\infty G_A\!\left(\|b(t)\|_{L_x^\infty}\right)\,dt < \infty, \qquad \text{for every} \quad A>0,
$$}

and a nonzero function
\[
\rho\in C^\infty((0,\infty)\times\tor)
\cap L^\infty((0,\infty)\times\tor)  \cap L^2((0,\infty);H^1(\tor)),
\]
with
\[
\lim_{t \to 0+} \|\rho(t)\|_{L^2} = 1,
\]
which solves \eqref{adv} with zero initial datum in the sense of \eqref{weak-form-intro}.
Hence, the zero solution is not unique within this class of distributional solutions.
\end{theorem}

We remark that the drifts in both Theorems \ref{main} and \ref{main2} satisfies 
\[
b\in\bigcap_{1\le q<2}L^q((0,\infty);L^\infty(\tor)).
\]
At large amplitudes, {$s^q=o(G_A(s))$}for every $q<2$, while {$G_A(s)=o(s^2)$}. Thus, our condition on the drift provides a single logarithmic refinement of $L^{2-}_t$. 

Now we consider a more restrictive weak-$L^2$ case for the drift. Surprisingly, even in this case positive-time energy balance (or even smoothness) does not imply uniqueness. While bounded energy solutions are unique, there are non-unique smooth for positive time solutions that blow up as $t \to 0^+$. More precisely, we have the following dichotomy.

\begin{theorem}[Weak-$L^2$ dichotomy]
\label{endpoint-main}
The following statements hold.
\begin{enumerate}
\item[\rm(i)] \emph{Non-uniqueness with positive-time energy equality.}
There exist a divergence-free drift
\[
b\in C^\infty((0,\infty)\times\tor)
\cap L_t^{2,\infty}L_x^\infty
\cap L_{t,x}^{2,\infty}
\]
and a nonzero function
\[
\rho\in C^\infty((0,\infty)\times\tor)
\cap\bigcap_{0<r\le\infty}L_t^{2,r}L_x^\infty
\]
such that $b\rho\in L_t^1L_x^\infty$ and $\rho$ solves \eqref{adv}
with zero initial datum in the sense of \eqref{weak-form-intro}.
Moreover, for every $0<t_0<t<\infty$,
\begin{equation}\label{endpoint-energy-constructed}
\frac12\|\rho(t)\|_{L^2}^2 + \int_{t_0}^t\|\nabla\rho(\tau)\|_{L^2}^2\,d\tau =
\frac12\|\rho(t_0)\|_{L^2}^2.
\end{equation}

\item[\rm(ii)] \emph{Global energy rigidity.}
For every $T>0$ and every divergence-free drift
\[
b\in L^{2,\infty}(0,T;L^\infty(\tor)),
\]
the zero solution is the unique distributional solution
\[
w\in L^\infty(0,T;L^2(\tor)) \cap L^2_{\rm loc}((0,T);H^1(\tor))
\]
of \eqref{adv} with zero initial datum satisfying the energy inequality
\begin{equation}\label{endpoint-energy-intro}
\frac12\|w(t)\|_{L^2}^2
+\int_{t_0}^t\|\nabla w(\tau)\|_{L^2}^2\,d\tau \le
\frac12\|w(t_0)\|_{L^2}^2,
\end{equation}
for every $t\in(t_0,T)$ and almost every $t_0\in(0,T)$.
\end{enumerate}
\end{theorem}

Note that in both part \textup{(i)} and part \textup{(ii)}, the drift belongs to $L_t^{2,\infty}L_x^\infty$. Consequently, the solution constructed in part \textup{(i)} cannot have bounded energy. In fact, the constructed solutions satisfy
\[
\lim_{t \to 0^+}\|\rho(t)\|_{L^2} = \infty.
\]

\subsection{Comparison with Previous Results}

The advection--diffusion equation with a rough incompressible drift has a long history of well-posedness results, alongside a more recent line of work exploring non-uniqueness. For a comprehensive overview of weak and parabolic solutions with integrable velocities, see the review by Bonicatto, Ciampa, and Crippa \cite{BonicattoCiampaCrippa2024}. 

\medskip
\noindent\textbf{The $L_t^2$ threshold and parabolic regularity.} 
In \cite{BonicattoCiampaCrippa2024}, it is shown that if $b\in L_t^2L_x^p$ and $\rho\in L_t^\infty L_x^q$ with $\frac1p+\frac1q\le \frac12$, then every distributional solution is parabolic and unique. (Sobolev drifts and enhanced dissipation-type estimates were studied, for example, by Bru\`e--Nguyen \cite{BrueNguyen2021}.) For spatially bounded drifts and bounded-in-time densities, strong $L_t^2$ is therefore the classical threshold for automatic parabolic regularity. Theorem \ref{main} and \ref{main2} demonstrate that, on the singular side of this boundary, nontrivial zero-data solutions can occur. 

\medskip
\noindent\textbf{Non-uniqueness by convex integration.}
\eqref{adv} can be viewed as a perturbation of the inviscid transport equation ($\partial_t\rho+b\cdot\nabla\rho=0$). For $b\in L^1_t(W^{1,p}_x\cap L^{q'}_x)$ and $\rho\in L^\infty_tL^q_x$, Modena and Sz\'ekelyhidi \cite{ModenaSzekelyhidi2018} proved transport non-uniqueness via convex integration when $\frac1p+\frac1q>1+\frac1{d-1}$, which was later improved to full-dimensional concentration $\frac1p+\frac1q>1+\frac1d$ by Modena and Sattig \cite{ModenaSattig2020}. Related critical-space transport examples were obtained in \cite{CheskidovLuo2021}. Because convex-integration schemes can handle the diffusive term when it is of lower order, these transport constructions can be adapted to produce examples for the advection--diffusion equation \eqref{adv}, requiring additionally $q'<d-1$ and $q'<d$, respectively. 

\medskip
\noindent\textbf{Non-uniqueness by stochastic Lagrangian flow.}
On $\mathbb T^2$, Moerschell and Sorella
\cite{MoerschellSorella2024} proved that, for every $r<2$, there
exists a drift in $L_t^rL_x^\infty$ and non-unique bounded parabolic solutions of \eqref{adv} satisfying a local energy inequality. Their proof uses a stochastic--Lagrangian mechanism. Theorems~\ref{main} and Theorems~\ref{main2} instead provides single drifts belonging simultaneously to every $L_t^rL_x^\infty$, $r<2$ (and even a logarithmic refinement), but the constructed nonzero solutions have infinite or finite energy jumps at the initial time. Theorem~\ref{endpoint-main} complements both results for drifts in $L^{2,\infty}_tL^2_x$.

In the proof of Theorem \ref{main} and \ref{endpoint-main}, we adapt a route fundamentally different from convex integration: rather than treating diffusion as a dominated error term to be ignored, diffusion is actively utilized as part of a frequency cascade.

\subsection{Overview of the Construction}

The solutions constructed in Theorems~\ref{main} and \ref{endpoint-main} are based on the instantaneous blow-up mechanism introduced by the first author, Dai, and Palasek \cite{CheskidovDaiPalasek2025} for the Navier--Stokes equations. In that setting, one builds a solution with zero initial data and size comparable to $t^{-1/2}$ as $t \to 0$. At a sequence of times $t_k$, an active component at frequency $\lambda_k$ creates a lower-frequency component at scale $\lambda_{k-1}$; the higher-frequency component is then rapidly damped by diffusion. Related inverse-cascade constructions have since appeared in magnetohydrodynamics \cite{Dai2026MHD}. For passive scalars, this perspective is also naturally adjacent to recent rigorous work on Obukhov--Corrsin anomalous dissipation and lack of selection by Colombo, Crippa, and Sorella \cite{ColomboCrippaSorella2023}.

For the first part, our construction is deliberately elementary. The drift and scalar are essentially one-directional on each time interval, and the building blocks are trigonometric functions with explicitly controlled Fourier support. This approach offers two major advantages. First, the principal cascade and its forcing can be estimated by hand. Second, the corrector can be constructed without invoking a general inverse-divergence operator. The frequency support is triangular, meaning the low-frequency leakage can be quantified directly. This triangular structure is the key feature that allows the projected recurrence estimate to avoid any false commutations between the Littlewood--Paley projection and the advection field.

While our main construction produces a diffusion-assisted inverse cascade of energy that blows up as $t \to 0$, we also explore an inverse-mixing mechanism where a finite amount of energy starts cascading from the infinite frequencies at the initial time. To this end we compress in time the inviscid perfect-mixing construction of
Alberti, Crippa, and Mazzucato \cite{ACM} so that the diffusion becomes perturbative. This produces a solution with bounded energy in Theorem~\ref{main2}.

\section{Settings}
Throughout this paper, coordinate indices are read modulo $d$:
\[
x_{j+d}=x_j,\qquad e_{j+d}=e_j,\qquad \partial_{j+d}=\partial_j.
\]

\subsection{A commutator estimate}

We use the following standard commutator estimate for $e^{t\Delta}$.

\begin{lemma}\label{ece}
Suppose that $\{\lambda_k\}_{k\ge1}$ is an increasing sequence of positive integers with $\lambda_{k+1}\ge10\lambda_k$ for all $k\ge1$. Then, for every $k\ge2$,
    $$
    \li{[e^{t\Delta},\prod_{i=1}^{k-1}\sin(\lambda_ix_i)]\sin(\lambda_kx_k)}\le C\lambda_{k-1}\lambda_k^{-1}e^{-\lambda_k^2t/4}.
    $$
Moreover, for every $k\ge 1$,
    $$
    \li{e^{t\Delta}\prod_{i=1}^k\sin(\lambda_ix_i)}\le Ce^{-\lambda_k^2t/4}.
    $$ 
    
    Arbitrarily many sines in the formula can be replaced with cosines.
\end{lemma}

\begin{proof}
    We have
    $$
    \prod_{i=1}^k\sin(\lambda_ix_i)=\frac1{(2\text{i})^{k}}\sum_{\xi\in\Omega_k}(-1)^{\sigma(\xi)}e^{\text{i}\xi\cdot x},
    $$
    where 
    $$
    \Omega_k=\{\sum_{i=1}^k\epsilon_i\lambda_ie_i:\epsilon\in\{-1,1\}^k\},
    $$
    and $\sigma(\xi)$ takes value from $\{0,1\}$. Then the commutator
    $$
    [e^{t\Delta},\prod_{i=1}^{k-1}\sin(\lambda_ix_i)]\sin(\lambda_kx_k)=\frac1{(2\text{i})^{k}}\sum_{\xi\in\Omega_k}(-1)^{\sigma(\xi)}(e^{-|\xi|^2t}-e^{-\lambda_k^2t})e^{\text{i}\xi\cdot x}.
    $$
    We notice that by $\lambda_{k+1}\ge10\lambda_k$ for all $k$
    $$
    \big||\xi|^2-\lambda_k^2\big|=\big|2\epsilon_k\lambda_ke_k\cdot\sum_{i=1}^{k-1}\epsilon_i\lambda_ie_i+|\sum_{i=1}^{k-1}\epsilon_i\lambda_ie_i|^2\big|\le C\lambda_{k-1}\lambda_k
    $$
    and $|\xi|^2\ge\frac{\lambda_k^2}2$.

    Then
    $$
    \begin{aligned}
        |e^{-|\xi|^2t}-e^{-\lambda_k^2t}|\le e^{-\min\{|\xi|,\lambda_k\}^2t}(1-e^{-||\xi|^2-\lambda_k^2|t})\le e^{-\lambda_k^2t/2}||\xi|^2-\lambda_k^2|t\\
        \le Ce^{-\lambda_k^2t/4}\frac{||\xi|^2-\lambda_k^2|}{\lambda_k^2}\le C\lambda_{k-1}\lambda_k^{-1}e^{-\lambda_k^2t/4},
    \end{aligned}
    $$
    where we used $e^{-\lambda_k^2t/4}t\le C\lambda_k^{-2}$.
    
    Finally, we see that
    $$
    \begin{aligned}
        &\li{[e^{t\Delta},\prod_{i=1}^{k-1}\sin(\lambda_ix_i)]\sin(\lambda_kx_k)}\\
        \le&\frac C{2^k}\sum_{\xi\in\Omega_k}\lambda_{k-1}\lambda_k^{-1}e^{-\lambda_k^2t/4}=C\lambda_{k-1}\lambda_k^{-1}e^{-\lambda_k^2t/4}.
    \end{aligned}
    $$

    For the second estimate, the case $k=1$ is immediate. For $k\ge2$,
    $$
    \begin{aligned}
        \li{e^{t\Delta}\prod_{i=1}^k\sin(\lambda_ix_i)}\le&\li{[e^{t\Delta},\prod_{i=1}^{k-1}\sin(\lambda_ix_i)]\sin(\lambda_kx_k)}\\
        &+\li{\prod_{i=1}^{k-1}\sin(\lambda_ix_i)e^{t\Delta}\sin(\lambda_kx_k)}\\
        \le&C\lambda_{k-1}\lambda_k^{-1}e^{-\lambda_k^2t/4}+e^{-\lambda_k^2t}\le Ce^{-\lambda_k^2t/4}.
    \end{aligned}
    $$
\end{proof}

\subsection{Choice of parameters}

We define
$$
s_k:=1+\frac{3}{k+2},\qquad Z_k:=1-\frac{1}{k+2}.
$$

Then $s_k \to 1$, $2-s_k<Z_k<1$. Define $\alpha_1=1$ and for $k\ge2$,
$$
\alpha_k=\frac12\left(1+\frac{(k+2)^2}{(k+1)(k+4)}\right)\alpha_{k-1}
=\left(1-\frac k{2(k+1)(k+4)}\right)\alpha_{k-1}.
$$
Then $\alpha_k\sim k^{-1/2} \to 0$.

Let
$$
\lambda_{k+1}=\lambda_k[\lambda_k^{s_k-1}],
$$
where $[\cdot]$ denotes the integer part and $\lambda_1\in\mathbb N$.

For
$\lambda_1$ sufficiently large,
\[
\left(1+\frac{3}{k+2}\right)\log\lambda_k-\log2
\le \log\lambda_{k+1}
\le \left(1+\frac{3}{k+2}\right)\log\lambda_k.
\]
Then
\[
\log\lambda_k\sim k^3,
\]
or
$$
\lambda_k\sim\lambda_1^{ck^3+O(k^2)}.
$$
The parameter $\lambda_1$ will be chosen in the following proposition.

Let
$$
\theta_k=\alpha_kZ_k,\qquad\theta_k^+=\frac{\alpha_k+\theta_k}{2}.
$$

\begin{proposition}\label{parameter}
    Let $\lambda_1$ be sufficiently large. Then, for every $k\ge2$, we may choose $m_k,N_k$ such that
    \begin{align}
        (\frac{2m_k}{e})^{m_k}\lambda_{k-1}^{-2}\lambda_k\lambda_{k+1}\le(\frac{\lambda_{k+1}}{\lambda_k})^{2m_k}\label{st1}\\
        (\frac{2m_k}{e})^{m_k}\lambda_{k-1}^{1-\alpha_k}\lambda_k^{\alpha_k}\lambda_{k+1}\le(\frac{\lambda_{k+1}}{\lambda_{k-1}})^{2m_k}\label{st2}\\
        (\frac{2m_k}{e})^{m_k}\lambda_{k-1}^{1-\alpha_k}\lambda_k^{1+\alpha_k}\le(\frac{\lambda_k}{\lambda_{k-1}})^{2m_k}\label{st3}\\
        C(\lambda_{k-1}^{\alpha_{k-1}+1}\lambda_k^{-1}+\lambda_{k-1}^{2\alpha_{k-1}}\lambda_k^{-\alpha_k})\le\lambda_{k-1}^{\theta_{k-1}}\label{st4}\\
        C\lambda_{k-1}^{\alpha_{k-1}}\lambda_k^{-\alpha_k+\theta_k}\le\lambda_{k-1}^{\theta_{k-1}}\label{st5}\\
        C(\frac{8N_k}e)^{N_k}\lambda_k^{\theta_k}2^k(\lambda_k^2t)^{-N_k}\le t^{-\theta_{k-1}/2},t\ge\lambda_k^{-2\theta_k/\theta_k^+}\label{st6}\\
        C(\frac{2N_k}e)^{N_k}\lambda_k^{\theta_k}2^k(\lambda_k^2\lambda_{k-1}^{-2})^{-N_k}\le\lambda_{k-1}^{\theta_{k-1}}\label{st7}
    \end{align}
\end{proposition}

{
\begin{proof}

    We take $m_k=k+2$. Then \eqref{st1}\eqref{st2}\eqref{st3} follow, up to harmless constants, from
    $$
    (\frac{2(k+2)}e)^{k+2}\lambda_k^{-2s_{k-1}^{-1}+1+s_k} \le
    \lambda_k^{2(s_k-1)(k+2)} = \lambda_k^6,
    $$
    $$
    (\frac{2(k+2)}e)^{k+2} \lambda_k^{(1-\alpha_k)s_{k-1}^{-1}+\alpha_k+s_k}
    \le \lambda_k^{2(k+2)(s_k-s_{k-1}^{-1})} =
    \lambda_k^{\frac{12(k+3)}{k+4}},
    $$
    $$
    (\frac{2(k+2)}e)^{k+2} \lambda_k^{(1-\alpha_k)s_{k-1}^{-1}+1+\alpha_k}
    \le \lambda_k^{2(k+2)(1-s_{k-1}^{-1})} = \lambda_k^{\frac{6(k+2)}{k+4}}.
    $$
    Since $\log\lambda_k\simeq_{\lambda_1}k^3$, the three inequalities automatically hold as long as $\lambda_1$ is large.

    For \eqref{st4}, we see that
    $$
    \lambda_{k-1}^{\alpha_{k-1}+1}\lambda_k^{-1}\lambda_{k-1}^{-\theta_{k-1}}\lesssim\lambda_{k-1}^{\frac{\alpha_{k-1}-3}{k+1}}\to0,
    $$

    $$
    \lambda_{k-1}^{2\alpha_{k-1}}\lambda_k^{-\alpha_k}
    \lambda_{k-1}^{-\theta_{k-1}} \lesssim
    \lambda_{k-1}^{\alpha_{k-1}\frac{k+2}{k+1} -\alpha_k\frac{k+4}{k+1}}
    = \lambda_{k-1}^{-\frac{3k+4}{2(k+1)^2}\alpha_{k-1}}
    \to0.
    $$
    Hence the two terms are smaller than $\frac1C$ as long as we take $\lambda_1$ to be large.

    For \eqref{st5}, we have
    $$
    \lambda_{k-1}^{\alpha_{k-1}-\theta_{k-1}} \lambda_k^{-\alpha_k+\theta_k}
    \lesssim \lambda_{k-1}^{\frac{\alpha_{k-1}}{k+1}
    -\frac{s_{k-1}\alpha_k}{k+2}} =
    \lambda_{k-1}^{-\frac{3k+4}{2(k+1)^2(k+2)}\alpha_{k-1}}
    \to 0.
    $$
    Thus \eqref{st5} holds once $\lambda_1$ is sufficiently large.

    Finally, take $N_k$ to be the smallest integer such that
    $$
    2N_k\ge\max\{d+1,2(2k+3)\}.
    $$
    Then $N_k=O(k)$ and, as long as $\lambda_1$ is large,
    $$
    C(\frac{8N_k}{e})^{N_k}2^k \le C_1(C_2k)^{C_3k}
    \le \lambda_k .
    $$
    Since $N_k>\theta_{k-1}/2$, it suffices in \eqref{st6} to take $t=\lambda_k^{-2\theta_k/\theta_k^+}$. At this time the needed exponent is
    $$
    2N_k\left(1-\frac{\theta_k}{\theta_k^+}\right) +
    \theta_k\left(\frac{\theta_{k-1}}{\theta_k^+}-1\right) >1,
    $$
    because $1-\theta_k/\theta_k^+=(2k+3)^{-1}$. Hence \eqref{st6} follows. The proof for \eqref{st7} is similar, using
    $$
    2N_k(1-s_{k-1}^{-1})+s_{k-1}^{-1}\theta_{k-1}-\theta_k>1.
    $$

\end{proof}
}

\subsection{Building Blocks}
Let
$$
\lambda_0=\alpha_0=1, \qquad \phi_0(x)=1, \qquad \rho_0=0.
$$
We define the potential
$$
\phi_ke_k=-\lambda_{k-1}^{1-\alpha_{k-1}}\lambda_k^{\alpha_k-1}\sin(\lambda_kx_k)\phi_{k-1}e_k, \qquad k\geq 1.
$$
Then take
$$
\bar\rho_k=-\lambda_{k-1}^{1-\alpha_{k-1}}\lambda_k^{\alpha_k} e^{-\lambda_k^2t}\cos(\lambda_kx_k)\phi_{k-1},
$$
and
$$
\rho_k=-\int_0^te^{(t-s)\Delta}\dv(b_{k+1}\bar\rho_{k+1})(s)\,ds,
$$
Choose a bump function $\eta\in C_c^\infty((0,\infty))$ with $\operatorname{supp}(\eta)\subset \left[1,2\right]$ such that
$$
\frac12\int_0^\infty\eta(s)e^{-s}\,ds=1.
$$

The incompressible vector field $b_k$ is defined by
$$
b_k=\lambda_{k-1}^{\alpha_{k-1}-1}\lambda_k^{2-\alpha_k}\eta(\lambda_k^2t)\cos(\lambda_kx_k)e_{k-1}.
$$
Define the global drift
$$
b:=\sum_{k\ge1}b_k.
$$
The support intervals
$$
J_k:=\left[\lambda_k^{-2},2\lambda_k^{-2}\right]
$$
are pairwise disjoint, because $\lambda_{k+1}>10\lambda_k$. Hence, at each $t>0$, at most one $b_k(t)$ is nonzero, so $b$ is well defined, divergence free, and belongs to $C^\infty((0,\infty)\times\tor)$. We call $J_k$ the activation interval of $b_k$, and we define
$$
I_k:=\left[\lambda_k^{-2},\lambda_{k-1}^{-2}\right]
$$
for the $k$-th time interval.

\begin{proposition}\label{b-lq}
For every $A>0$,
\[
\int_0^\infty G_A\!\left(\|b(t)\|_{L_x^\infty}\right)\,dt<\infty.
\]
Consequently,
\[
b\in\bigcap_{1\le q<2}L^q((0,\infty);L^\infty(\tor)).
\]
\end{proposition}
\begin{proof}
Let
\[
L_k:=\log\lambda_k,
\qquad
T_k:=(1-\alpha_k)L_k,
\qquad
a_k:=\lambda_{k-1}^{\alpha_{k-1}-1}\lambda_k^{2-\alpha_k}.
\]
Then
\[
a_k=\lambda_k e^{T_k-T_{k-1}}.
\]

Since $L_k\sim k^3$, by definitions of $\lambda_k$ and
$\alpha_k$ we get
\[
0\le L_k-L_{k-1}
\le \frac{3}{k+1}L_{k-1}=O(k^2)
\]
and
\[
0\le\alpha_{k-1}-\alpha_k =\frac{k}{2(k+1)(k+4)}\alpha_{k-1}
=O(k^{-3/2}).
\]
Consequently,
\[
0\le T_k-T_{k-1}
=(1-\alpha_k)(L_k-L_{k-1})
 +(\alpha_{k-1}-\alpha_k)L_{k-1}
\le Ck^2.
\]
Let $C_\eta:=\|\eta\|_{L^\infty}$. Since $a_k\ge\lambda_k$ and
$L_k\asymp k^3$,
\[
\bigl(\log(e+C_\eta a_k)\bigr)^{3/4}\ge ck^{9/4}
\]
for all sufficiently large $k$. Moreover, $G_A(s)\to\infty$ as
$s\to\infty$, and $G_A$ is increasing for all sufficiently large $s$.
Hence, after discarding finitely many indices,
\[
\sup_{0\le s\le C_\eta a_k}G_A(s)\le G_A(C_\eta a_k).
\]
Using the disjointness of the intervals $J_k$ and absorbing finitely many
indices into the constant, we obtain
\[
\begin{aligned}
\int_0^\infty G_A\!\left(\|b(t)\|_{L_x^\infty}\right)\,dt
&\le C+\sum_{k\ge2}\lambda_k^{-2}G_A(C_\eta a_k)\\
&\le C+C\sum_{k\ge2}
\exp\!\left(2(T_k-T_{k-1})-cAk^{9/4}\right)<\infty
\end{aligned}
\]
for any $A>0$.

For every $1\le q<2$,
\[
s^q\le C_{A,q}\bigl(1+G_A(s)\bigr),
\qquad s\ge0.
\]
Since $b=0$ outside $\bigcup_kJ_k$ and
$\sum_k|J_k|<\infty$, the stated $L_t^qL_x^\infty$ bounds follow.
\end{proof}

\section{Basic Properties}

\begin{proposition}\label{phi}
    It holds that
    $$
    \li{\nabla^m\phi_k}\le2^m\lambda_k^{m-1+\alpha_k},
    $$
    and
    $$
    \li{\partial_{k-1}\phi_k}\le2\lambda_{k-1}\lambda_k^{\alpha_k-1}.
    $$
\end{proposition}

\begin{proof}
    By definition,
    $$
    \begin{aligned}
        \phi_k=(-1)^k\lambda_k^{\alpha_k-1}\prod_{i=1}^k\sin(\lambda_ix_i)
    =(-1)^k\lambda_k^{\alpha_k-1}(2\mathrm i)^{-k}\sum_{\xi\in\Omega_k}(-1)^{\sigma(\xi)}\epsilon_ie^{\mathrm i\xi\cdot x},
    \end{aligned}
    $$
    where $\Omega_k=\{\sum_{i=1}^k\epsilon_i\lambda_ie_i:\epsilon\in\{-1,1\}^k\}$. Since each individual frequency satisfies
    $$
    |\xi|=|\sum_{i=1}^k\epsilon_i\lambda_ie_i|\le \sum_{i=1}^k\lambda_i<2\lambda_k,
    $$
    we obtain
    $$
    \li{\nabla^m\phi_k}
    \le \lambda_k^{\alpha_k-1}2^{-k}\sum_{\xi\in\Omega_k}(2\lambda_k)^m
    \le 2^m\lambda_k^{m+\alpha_k-1}.
    $$
    For the directional derivative, since $\partial_{k-1}$ doesn't act on $\sin(\lambda_kx_k)$, we have
    $$
    |\prod_{1\le i\le k}\partial_{k-1}e^{\mathrm{i}\sum_{i=1}^k\epsilon_i\lambda_ix_i}|\le\lambda_{k-1}+\lambda_{k-1-d}+\cdots\le 2\lambda_{k-1}.
    $$
    Therefore, we have
    $$
    \li{\partial_{k-1}\phi_k}\le2\lambda_{k-1}\lambda_k^{\alpha_k-1}.
    $$
\end{proof}

We claim the following estimates for $\rho_k$.
\begin{proposition}\label{key}
    \begin{enumerate}[label=\textup{(\roman*)}]
        \item $\lambda_k\li{\rho_k}+\li{\partial_k\rho_k}\le C\lambda_k^{1+\alpha_k}$ for $\lambda_{k+1}^{-2}\le t\le 2\lambda_{k+1}^{-2}$.
        \item $\li{\partial_{k-1}(\rho_k-\bar\rho_k)}+\lambda_{k-1}\li{\rho_k-\bar\rho_k}\le C\lambda_{k-1}^2\lambda_k^{\alpha_k-1}$ for $t\ge\lambda_k^{-2}$.
        \item $\li{\rho_k}+\li{\nabla\rho_k}\le C\lambda_{k-1}^{\alpha_k-1}$ for $\lambda_{k-1}^{-2}\le t\le1$ and $\li{\rho_k}+\li{\nabla\rho_k}\le Ce^{-\lambda_k^2t/4}$ for $t\ge1$.
        \item $\mathrm{supp}(\mathcal F\rho_k)\subset\Omega_{k+1}+\{-\lambda_{k+1}e_{k+1},\lambda_{k+1}e_{k+1}\}\subset\Omega_k+\lambda_{k+1}\mathbb Z^d$, where $\Omega_k$ is as in Proposition \ref{phi}.
    \end{enumerate}
\end{proposition}

\begin{proof}
    Since the heat semigroup does not change the Fourier support, (iv) is immediate from the explicit formula for $\rho_k$.

    Write
    $$
    K_{k+1}(s):=\frac12\lambda_{k+1}^2\eta(\lambda_{k+1}^2s)e^{-\lambda_{k+1}^2s}.
    $$
    Then $\int_0^\infty K_{k+1}(s)\,ds=1$, and $K_{k+1}$ is supported on $[\lambda_{k+1}^{-2},2\lambda_{k+1}^{-2}]$.
    
    Since $d\ge2$, the derivative in $\operatorname{div}(b_{k+1}\bar\rho_{k+1})$ falls only on $\phi_k$. Thus
    $$
    \rho_k(t)=\int_0^t K_{k+1}(s)e^{(t-s)\Delta}\Bigl[\bigl(1+\cos(2\lambda_{k+1}x_{k+1})\bigr)\partial_k\phi_k\Bigr]\,ds,
    $$
    as well as
    \begin{equation}\label{p_k_rho}
        \partial_k\rho_k(t)=\int_0^tK_{k+1}(s)e^{(t-s)\Delta}\Big(\big(1+\cos(2\lambda_{k+1}x_{k+1})\big)\partial_k^2\phi_ke_k\Big)\,ds.
    \end{equation}

    For (i), \eqref{p_k_rho} and the $L^\infty$-contractivity of the heat semigroup give
    $$
    \li{\partial_k\rho_k(t)}\le\int_0^tK_{k+1}(s)\li{\big(1+\cos(2\lambda_{k+1}x_{k+1})\big)\partial_k^2\phi_ke_k}\,ds.
    $$
    Proposition \ref{phi} therefore yields
    $$
    \li{\partial_k\rho_k(t)}\le C\lambda_k^{1+\alpha_k},
    $$
    and the same argument works for $\lambda_k\li{\rho_k}$.

    For (ii), we only prove the bound for $\li{\partial_{k-1}(\rho_k-\bar\rho_k)}$ in case that $\lambda_{k-1}\li{\rho_k-\bar\rho_k}$ can be estimated in the same way. We decompose the forcing into its constant Fourier mode and the oscillatory remainder. Then we obtain
    $$
    \partial_{k-1}\rho_k(t)=I+II,
    $$
    where
    $$
    I:=\int_0^t K_{k+1}(s)e^{(t-s)\Delta}\partial_{k-1}\partial_k\phi_k\,ds,
    $$
    and
    $$
    \begin{aligned}
        II:=E_1+E_2=\int_0^tK_{k+1}(s)e^{(t-s)\Delta}(\cos(2\lambda_{k+1}x_{k+1})\partial_{k-1}\partial_k\phi_k)\,ds\\
        +\int_0^tK_{k+1}(s)e^{(t-s)\Delta}(\partial_{k-1}\cos(2\lambda_{k+1}x_{k+1})\partial_k\phi_k)\,ds.
    \end{aligned}
    $$
    The integrand in $E_2$ is zero when $d>2$. For $d=2$, $E_1$ is smaller than $E_2$ because $\partial_{k-1}$ acts on $\phi_k$ in $E_1$, whereas it acts on the higher frequency factor $\cos(\lambda_{k+1}x_{k+1})$ in $E_2$. It suffice to estimate $E_2$ for $d=2$. We have
    $$
    \begin{aligned}
        &\partial_{k-1}\cos(2\lambda_{k+1}x_{k+1})\partial_k\phi_k\\
        =&-2\lambda_{k+1}\lambda_k^{\alpha_k-1}\sin(2\lambda_{k+1}x_{k+1})\Big(\lambda_k\cos(\lambda_kx_k)\prod_{1\le i\le k,i\ne k}\sin(\lambda_ix_i)\\
        &+\lambda_{k-d}\cos(\lambda_{k-d}x_k)\prod_{1\le i\le k,i\ne k-d}\sin(\lambda_ix_i)+\cdots\Big).
    \end{aligned}
    $$

    Then by Lemma \ref{ece}, we have
    $$
    \begin{aligned}
        &\li{e^{t\Delta}(\partial_{k-1}\cos(2\lambda_{k+1}x_{k+1})\partial_k\phi_k)}\\
        \le&C\lambda_{k+1}\lambda_k^{\alpha_k-1}(\lambda_k+\lambda_{k-d}+\cdots)e^{-\lambda_{k+1}^2t}\le C\lambda_{k+1}\lambda_k^{\alpha_k}e^{-\lambda_{k+1}^2t}
    \end{aligned}
    $$
    since $\lambda_k+\lambda_{k-d}+\cdots<2\lambda_k$. This implies
    $$
    |E_2|\le C\lambda_{k+1}\lambda_k^{\alpha_k}e^{-\lambda_{k+1}^2t/2}\int_0^tK_{k+1}(s)\,ds\le C\lambda_{k+1}\lambda_k^{\alpha_k}e^{-\lambda_{k+1}^2t/2}
    $$
    because $t-s\ge\frac t2$ when $t\ge\lambda_k^{-2}$ and $s\in\operatorname{supp}K_{k+1}$. We further have
    $$
    |E_2|\le C(\frac{2m_k}e)^{m_k}\lambda_k^{\alpha_k+2m_k}\lambda_{k+1}^{1-2m_k}
    $$
    by $e^{-x}\le(\frac me)^mx^{-m}$.

    The same argument yields a smaller bound for $E_1$. By \eqref{st1}, we see that $|II|\le C\lambda_{k-1}^2\lambda_k^{\alpha_k-1}$.
    
    Next, we show that $I$ is close to $\partial_k\bar\rho_k$. By Lemma \ref{phi},
    $$
    \begin{aligned}
        \li{I-e^{t\Delta}\partial_{k-1}\partial_k\phi_k}
        &\le \int_0^\infty K_{k+1}(s)\li{(e^{(t-s)\Delta}-e^{t\Delta})\partial_{k-1}\partial_k\phi_k}\,ds\\
        &\le C\lambda_{k+1}^{-2}\li{\Delta\partial_{k-1}\partial_k\phi_k}\\
        &\le C\lambda_{k+1}^{-2}\lambda_k^{2+\alpha_k}\lambda_{k-1},
    \end{aligned}
    $$
    because the support of $K_{k+1}$ has length $\lambda_{k+1}^{-2}$. Our choice of $s_k$ guarantees that $2s_k+s_{k-1}^{-1}>3$, which means $\lambda_{k+1}^{-2}\lambda_k^{2+\alpha_k}\lambda_{k-1}\le\lambda_{k-1}^2\lambda_k^{\alpha_k-1}$. So we further have
    $$
    \li{I-e^{t\Delta}\partial_{k-1}\partial_k\phi_k}\le C\lambda_{k-1}^2\lambda_k^{\alpha_k-1}.
    $$

    Finally, using the formulas for $\phi_k$ and $\bar\rho_k$,
    $$
    \begin{aligned}
        &\li{e^{t\Delta}\partial_{k-1}\partial_k\phi_k-\partial_{k-1}\bar\rho_k}\\
        \le&C(\lambda_{k-1}^{1-\alpha_{k-1}}\lambda_k^{\alpha_k}\li{[e^{t\Delta},\partial_{k-1}\phi_{k-1}]\cos(\lambda_kx_k)}\\
        &+\lambda_{k-1}^{1-\alpha_{k-1}}\lambda_k^{\alpha_k-1}\li{\sin(\lambda_kx_k)\partial_{k-1}\partial_k\phi_{k-1}}).
    \end{aligned}
    $$

    By Proposition \ref{phi}, the second term is smaller than $\lambda_{k-1}^2\lambda_k^{\alpha_k-1}$. For the first term, we plug in the explicit expression of $\partial_{k-1}\phi_{k-1}$

$$
\partial_{k-1}\phi_{k-1}
=
(-1)^{k-1}\lambda_{k-1}^{\alpha_{k-1}-1}
\sum_{\substack{1\le r\le k-1\\ r\equiv k-1\pmod d}}
\lambda_r\cos(\lambda_rx_r)
\prod_{\substack{1\le i\le k-1\\ i\ne r}}
\sin(\lambda_ix_i).
$$

    and apply Lemma \ref{ece} to each summand to get
    $$
    \begin{aligned}
        \li{[e^{t\Delta},\partial_{k-1}\phi_{k-1}]\cos(\lambda_kx_k)}\le&C\lambda_{k-1}^{\alpha_{k-1}}\lambda_k^{-1}(\lambda_{k-1}+\lambda_{k-1-d}+\cdots)\\
        \le&C\lambda_{k-1}^{\alpha_{k-1}+1}\lambda_k^{-1}.
    \end{aligned}
    $$
    Therefore, we have
    $$
    \li{e^{t\Delta}\partial_{k-1}\partial_k\phi_k-\partial_{k-1}\bar\rho_k}\le\lambda_{k-1}^2\lambda_k^{\alpha_k-1}.
    $$
    This completes the estimates for $\li{\partial_{k-1}(\rho_k-\bar\rho_k)}$.
    
    The proof of (iii) is analogical. For $j=1,\dots,d$, we decompose $\partial_j\rho_k=I+II$ as
    $$
    I=\int_0^tK_{k+1}(s)e^{(t-s)\Delta}\partial_j\partial_k\phi_k\,ds
    $$
    and
    $$
    \begin{aligned}
        II=E_1+E_2=&\int_0^tK_{k+1}(s)e^{(t-s)\Delta}(\cos(2\lambda_{k+1}x_{k+1})\partial_j\partial_k\phi_k)\,ds\\
        &+\int_0^tK_{k+1}(s)e^{(t-s)\Delta}(\partial_j\cos(2\lambda_{k+1}x_{k+1})\partial_k\phi_k)\,ds.
    \end{aligned}
    $$
    As before, we deduce that
    $$
    |II|\le C\lambda_{k+1}\lambda_k^{\alpha_k}e^{-\lambda_{k+1}^2t/2}.
    $$
    Since $t\ge\lambda_{k-1}^{-2}$, we have
    $$
    |II|\le C(\frac{m_k}e)^{m_k}\lambda_{k+1}^{1-2m_k}\lambda_k^{\alpha_k}\lambda_{k-1}^{2m_k}\le C\lambda_{k-1}^{\alpha_k-1}
    $$
    by \eqref{st2}. Moreover, when $t\ge1$, we simply have
    $$
    |II|\le C\lambda_{k+1}^{-1}\lambda_k^{\alpha_k}e^{-\lambda_{k+1}^2t/4}\le Ce^{-\lambda_k^2t/4}.
    $$

    To estimate $I$, we expand $\partial_j\partial_k\phi_k$ as the combination of product of $\sin(\lambda_ix_i)$ and $\cos(\lambda_ix_i)$ as before. Applying Lemma \ref{ece}, we obtain
    $$
    |I|\le C\lambda_k^{\alpha_k-1}\int_0^tK_{k+1}(s)e^{-\lambda_k^2(t-s)/2}\sum_{1\le m,n\le k}\lambda_m\lambda_n\,ds\le C\lambda_k^{\alpha_k+1}e^{-\lambda_k^2t/2}
    $$
    for $t\ge\lambda_{k-1}^{-2}$. Therefore,
    $$
    |I|\le C(\frac{m_k}e)^{m_k}\lambda_k^{\alpha_k+1-2m_k}\lambda_{k-1}^{2m_k}\le C\lambda_{k-1}^{\alpha_k-1}
    $$
    because of \eqref{st3}.
    
    For $t\ge1$, we further have $|I|\le C\lambda_k^{\alpha_k-1}e^{-\lambda_k^2t/4}\le Ce^{-\lambda_k^2t/4}$. We conclude the estimate for $|\nabla\rho_k|$. Since $\rho_k$ solves a divergence-form equation with zero initial data, it has zero spatial mean for every $t$. Therefore,
    $$
    \|\rho_k(t)\|_{L^\infty}\le C\|\nabla\rho_k(t)\|_{L^\infty}
    $$
    and this completes the proof of (iii).
\end{proof}

\section{Estimate for the Principal Part}
Define the principal part of the solution
$$
\varrho=\sum_k\rho_k.
$$

\begin{proposition}\label{principal-bounds}
    $$
    \li{\varrho(t)}\le Ct^{-\alpha_k/2}\qquad \text{for all }\lambda_{k+1}^{-2}\le t\le\lambda_k^{-2},
    $$
    and
    $$
    \li{\varrho(\lambda_k^{-2})}\ge C\lambda_k^{\alpha_k}.
    $$
\end{proposition}
\begin{proof}
    Since $t\le\lambda_k^{-2}$, $\rho_j(t)=0$ for every $j\le k-1$. We use the definition of $\rho_k$, the support of $b_{k+1}$, the $L^\infty$-contractivity of the heat semigroup, and Proposition \ref{phi}, to obtain
    $$
    \begin{aligned}
        &\li{\rho_k(t)}\le \int_{J_{k+1}}\|\dv(b_{k+1}\bar\rho_{k+1})(s)\|_{L^\infty}\,ds\\
        \le&C\|\partial_k\phi_k\|_{L^\infty}\le C\lambda_k^{\alpha_k}\le Ct^{-\alpha_k/2}.
    \end{aligned}
    $$

Let $\varepsilon_0:=1-\alpha_2>0$. Since $\alpha_j\le\alpha_2$ for $j\ge2$, Proposition~\ref{key} (iii) and lacunarity imply
\[
\sum_{j\ge k+2}\|\rho_j(t)\|_{L^\infty} \le
C\sum_{j\ge k+2}\lambda_{j-1}^{-\varepsilon_0} \le C
\le Ct^{-\alpha_k/2},
\]
because $t\le1$. Then we estimate $\rho_{k+1}$. Since
$$
1-\alpha_k-s_k(1-\alpha_{k+1})=-\frac3{k+2}(1-\alpha_{k+1})-s_k(\alpha_k-\alpha_{k+1})<-\frac{3}{2(k+2)},
$$
we deduce that $\lambda_k^{1-\alpha_k}\lambda_{k+1}^{\alpha_{k+1}-1}\ll1$, and hence by Proposition \ref{key}
\[
\|\rho_{k+1}(t)-\bar\rho_{k+1}(t)\|_{L^\infty} \le C\lambda_k\lambda_{k+1}^{\alpha_{k+1}-1}
    \ll C\lambda_k^{\alpha_k} \le Ct^{-\alpha_k/2}.
\]
Together with
\[
\|\bar\rho_{k+1}(t)\|_{L^\infty} \le C\lambda_{k+1}^{\alpha_{k+1}}e^{-\lambda_{k+1}^2t}
    \le Ct^{-\alpha_{k+1}/2} \ll Ct^{-\alpha_k/2}
\]
for $\lambda_{k+1}^{-2}\le t\le\lambda_k^{-2}$, we see that $\li{\rho_{k+1}(t)}$ is much smaller than $t^{-\alpha_k/2}$.

Combining the direct estimate for $\rho_k$ above with the estimates for
$\rho_{k+1}$ and the tail, we obtain
$$
\begin{aligned}
\li{\varrho(t)}
&\le \li{\bar\rho_k(t)} + \li{\bar\rho_k(t)-\rho_k(t)} +
\li{\rho_{k+1}(t)} + \sum_{j\ge k+2}\li{\rho_j(t)}\\
&\le Ct^{-\alpha_k/2}.
\end{aligned}
$$

    We then estimate the lower bound.

    By Proposition \ref{key}, we have
    $$
    \li{\rho_k(\lambda_k^{-2})-\bar\rho_k(\lambda_k^{-2})}\le C\lambda_{k-1}\lambda_k^{\alpha_k-1}
    $$
    which is much smaller than $\lambda_k^{\alpha_k}$.
    
    Also, after increasing $\lambda_1$ if necessary, by super-lacunarity,
$$
\|
\cos(\lambda_kx_k)\prod_{i=1}^{k-1}\sin(\lambda_ix_i) \|_{L^\infty}
\ge c
$$
for some constant $c>0$. Then
    $$
    \li{\bar\rho_k(\lambda_k^{-2})}\ge C\lambda_k^{\alpha_k}.
    $$

    Finally,
    $$
    \begin{aligned}
        &\li{\varrho(\lambda_k^{-2})}\ge\li{\bar\rho_k(\lambda_k^{-2})}-\li{\bar\rho_k(\lambda_k^{-2})-\rho_k(\lambda_k^{-2})}\\
        &-\li{\rho_{k+1}(\lambda_k^{-2})}-\sum_{j>k}\li{\rho_j(\lambda_k^{-2})}\ge C\lambda_k^{\alpha_k}.
    \end{aligned}
    $$
    
\end{proof}

\begin{corollary}\label{principal-lp}
    The principal part satisfies
    $$
    \varrho\in\bigcap_{1\le p<\infty}L^p((0,\infty);L^\infty(\tor)).
    $$
\end{corollary}
\begin{proof}
    Fix $p<\infty$. Since $\alpha_k\to0$, there exists $i_0$ such that $p\alpha_{i_0}<2$ and hence $p\alpha_i<2$ for every $i\ge i_0$. On each later block interval $[\lambda_{i+1}^{-2},\lambda_i^{-2}]$, Proposition \ref{principal-bounds} implies
    $$
    \|\varrho(t)\|_{L^\infty}^p\le Ct^{-p\alpha_i/2}\le Ct^{-p\alpha_{i_0}/2}.
    $$
    Therefore
    $$
    \sum_{i\ge i_0}\int_{\lambda_{i+1}^{-2}}^{\lambda_i^{-2}}\|\varrho(t)\|_{L^\infty}^p\,dt
    \le C\int_0^{\lambda_{i_0}^{-2}} t^{-p\alpha_{i_0}/2}\,dt<\infty.
    $$
    The finitely many earlier block intervals also contribute finitely because they stay away from $t=0$.
    Hence $\varrho\big|_{t\le 1}\in L^p((0,1);L^\infty(\tor))$ for every finite $p$. Thanks to the exponential decay of $\rho_k$ for $t\ge1$, as is given by Proposition \ref{key}(iii), we see $\varrho\big|_{t\ge1}\in L^p((1,\infty);L^\infty(\tor))$.
\end{proof}

\section{Construction of Corrector}
By definition, $\rho_k$ satisfies
$$
\partial_t\rho_k+\dv(b_{k+1}\bar\rho_{k+1})-\Delta\rho_k=0.
$$

Formally summing in $k$, we obtain
\begin{equation}\label{fad}
\begin{aligned}
\partial_t\varrho+b\cdot\nabla\varrho-\Delta\varrho&=\sum_{j,k}\dv(b_k\rho_j)-\sum_k\dv(b_k\bar\rho_k)\\
&=\sum_kb_k\cdot\nabla\Bigl(\sum_{j>k}\rho_j+(\rho_k-\bar\rho_k)+\rho_{k-1}\Bigr),
\end{aligned}
\end{equation}
because $\rho_j$ is created only after $\lambda_{j+1}^{-2}$. We will later justify \eqref{fad} rigorously in $\mathcal D'([0,\infty)\times\tor)$.

We next introduce approximate correctors for the same equation, obtained from truncated approximation problems.

We define $\rho^n$ to be the solution of
$$
\begin{cases}
    \partial_t\rho^n+b\cdot\nabla\rho^n-\Delta\rho^n=\sum_{k\le n}f_k&\quad t\ge\lambda_n^{-2}\\
    \rho^n=0&\quad 0\le t\le\lambda_n^{-2}
\end{cases}
$$
where
$$
f_k=b_k\cdot\nabla(\sum_{j>k}\rho_j+(\rho_k-\bar\rho_k)+\rho_{k-1}).
$$
For later use, take a distribution $f$ such that
$$
\langle f,\varphi\rangle=-\int_0^\infty\int_{\tor}\sum_{k\ge1}b_k(\sum_{j>k}\rho_j+(\rho_k-\bar\rho_k)+\rho_{k-1})\cdot\nabla\varphi\,dx\,dt,
$$
and we formally have
$$
f=\sum_{k\ge1}f_k.
$$
We will show the convergence later. Since each $f_k$ is supported in $J_k$, the sum defining $f$ is locally finite in time on $(0,1]\times\tor$.

Away from $t=0$, each $\rho^n$ is globally defined, spatially smooth, and smooth in time. The next proposition gives the one-interval recurrence mechanism on the interior time intervals of a fixed block.

\begin{lemma}[Iteration]\label{boundary-iteration}
    Let $\rho$ be the solution to
    $$
    \begin{cases}
        \partial_t\rho+b_k\cdot\nabla\rho-\Delta\rho=f_k\quad&\lambda_k^{-2}\le t\le\lambda_{k-1}^{-2},\\
        \rho(\lambda_k^{-2})=\rho_{\mathrm{in}}
    \end{cases}
    $$
    with
    $$
    \mathrm{supp}(\mathcal F\rho_{\mathrm{in}})\subset\Omega_k+\lambda_{k+1}\mathbb Z^d, \qquad
    \|\rho_{\mathrm{in}}\|_{L^\infty}\le\lambda_k^{\theta_k}.
    $$
    Then, for every $\lambda_k^{-2}\le t\le\lambda_{k-1}^{-2}$,
    $$
    \mathrm{supp}(\mathcal F\rho(t))\subset\Omega_{k-1}+\lambda_k\mathbb Z^d,
    $$
    and
    $$
    \|\rho(t)\|_{L^\infty}\le Ct^{-\theta_k^+/2},
    \qquad
    \|\rho(\lambda_{k-1}^{-2})\|_{L^\infty}\le\lambda_{k-1}^{\theta_{k-1}}.
    $$
\end{lemma}

\begin{proof}
    Let $\rho=\rho^1+\rho^2$ with
    $$
    \begin{cases}
        \partial_t\rho^1+b_k\cdot\nabla\rho^1-\Delta\rho^1=f_k\quad&\lambda_k^{-2}\le t\le\lambda_{k-1}^{-2},\\
        \rho^1(\lambda_k^{-2})=0
    \end{cases}
    $$
    and
    $$
    \begin{cases}
        \partial_t\rho^2+b_k\cdot\nabla\rho^2-\Delta\rho^2=0\quad&\lambda_k^{-2}\le t\le\lambda_{k-1}^{-2},\\
        \rho^2(\lambda_k^{-2})=\rho_{\mathrm{in}}.
    \end{cases}
    $$

    We first estimate the forced part. By the maximum principle, Proposition \ref{key}(ii), and Proposition \ref{key}(i) applied to $\rho_{k-1}$,
    \begin{equation*}
    \begin{aligned}
        \|\rho^1(t)\|_{L^\infty}
        &\le\int_{J_k}\|b_k\cdot\nabla(\rho_k-\bar\rho_k)\|_{L^\infty}\,ds
        +\int_{J_k}\|b_k\cdot\nabla\rho_{k-1}\|_{L^\infty}\,ds\\
        &+\sum_{k<j}\int_{J_k}\|b_k\cdot\nabla\rho_j\|_{L^\infty}\,ds\\
        &\le C(\lambda_{k-1}^{\alpha_{k-1}+1}\lambda_k^{-1}
        +\lambda_{k-1}^{2\alpha_{k-1}}\lambda_k^{-\alpha_k}+\sum_{k<j}\lambda_k^{-\alpha_k}\lambda_{k-1}^{\alpha_{k-1}-1}\lambda_{j-1}^{\alpha_j-1})\\
        &\le C(\lambda_{k-1}^{\alpha_{k-1}+1}\lambda_k^{-1}
        +\lambda_{k-1}^{2\alpha_{k-1}}\lambda_k^{-\alpha_k}).
    \end{aligned}
    \end{equation*}
    The last inequality is because the contribution of $\sum_{k<j}b_k\cdot\nabla\rho_j$ is even smaller than $1$ and thus can be absorbed by other terms. Then by \eqref{st4},
    \begin{equation}\label{bi1}
        \|\rho^1(t)\|_{L^\infty}\le\frac13\lambda_{k-1}^{\theta_{k-1}}.
    \end{equation}
    Since $\theta_k^+>\theta_{k-1}$ and $t\le\lambda_{k-1}^{-2}$ on the time interval, this yields
    $$
    \|\rho^1(t)\|_{L^\infty}\le Ct^{-\theta_k^+/2}
    \qquad\text{and}\qquad
    \|\rho^1(\lambda_{k-1}^{-2})\|_{L^\infty}\le\frac13\lambda_{k-1}^{\theta_{k-1}}.
    $$

    Next we estimate the Fourier support. Let
    $$
    E_{k-1}:=\Omega_{k-1}+\lambda_k\mathbb Z^d.
    $$
    By Proposition \ref{key}(iv), the explicit formula for $\bar\rho_k$, and the expression of $b_k$, the force $b_k\cdot\nabla(\sum_{k<j}\rho_j+(\rho_k-\bar\rho_k)+\rho_{k-1})$ is supported in $E_{k-1}$. The initial support assumption also implies $\mathrm{supp}(\mathcal F\rho_{\mathrm{in}})\subset E_{k-1}$, because $\lambda_{k+1}\in\lambda_k\mathbb N$ and $\pm\lambda_ke_k\in\lambda_k\mathbb Z^d$. Moreover, $E_{k-1}$ is invariant under diffusion and under the operator $f\mapsto b_k\cdot\nabla f$: multiplying by $\cos(\lambda_kx_k)$ only shifts the Fourier support by $\pm\lambda_ke_k\in\lambda_k\mathbb Z^d$, while $\partial_{k-1}$ does not enlarge the support. Therefore
    $$
    \mathrm{supp}(\mathcal F\rho^m(t))\subset E_{k-1},\qquad m=1,2,
    $$
    and hence $\mathrm{supp}(\mathcal F\rho(t))\subset E_{k-1}$.

    We next estimate the homogeneous part $\rho^2$. Let $\mathbf P_{\le 2\lambda_{k-1}}$ be a smooth Fourier multiplier whose symbol is equal to $1$ on $B(0,2\lambda_{k-1})$ and supported in $B(0,4\lambda_{k-1})$, and set $\mathbf P_{>2\lambda_{k-1}}=I-\mathbf P_{\le 2\lambda_{k-1}}$. Since $\lambda_k>10\lambda_{k-1}$ and $\sum_{j=1}^{k-1}\lambda_j<2\lambda_{k-1}$, every frequency
    $$
    \xi=\epsilon_1\lambda_1e_1+\cdots+\epsilon_{k-1}\lambda_{k-1}e_{k-1}+\lambda_kz\in E_{k-1}
    $$
    either has $z=0$ and then $|\xi|<2\lambda_{k-1}$, or has $z\ne0$ and then
    $$
    |\xi|\ge\lambda_k|z|-\sum_{j=1}^{k-1}\lambda_j>\frac12\lambda_k>4\lambda_{k-1}.
    $$
    Thus the low and high sectors of $E_{k-1}$ are cleanly separated by this cutoff. 
    
    Let $w=\mathbf P_{\le 2\lambda_{k-1}}\rho^2$. Since $\nabla\cdot b_k=0$, we have
    $$
    \partial_tw-\Delta w=-\mathbf P_{\le 2\lambda_{k-1}}(b_k\cdot\nabla\rho^2)=-\dv\mathbf P_{\le 2\lambda_{k-1}}(b_k\rho^2)
    $$
    on $I_k$. Moreover, the initial support assumption places every frequency of $\rho_{\mathrm{in}}$ in the high sector, so
    $$
    w(\lambda_k^{-2})=\mathbf P_{\le 2\lambda_{k-1}}\rho_{\mathrm{in}}=0.
    $$
    
    By Duhamel's formula and the $L^\infty$-contractility of the heat semigroup,
    \begin{equation*}
    \begin{aligned}
        \|w(t)\|_{L^\infty}
        &\le C\int_{J_k}\lambda_{k-1}\|b_k(s)\rho^2(s)\|_{L^\infty}\,ds\\
        &\le C\lambda_{k-1}^{\alpha_{k-1}}\lambda_k^{-\alpha_k}\|\rho_{\mathrm{in}}\|_{L^\infty}\\
        &\le C\lambda_{k-1}^{\alpha_{k-1}}\lambda_k^{-\alpha_k+\theta_k}.
    \end{aligned}
    \end{equation*}
    By \eqref{st5},
    \begin{equation}\label{bi2}
        \|w(t)\|_{L^\infty}\le\frac13\lambda_{k-1}^{\theta_{k-1}}.
    \end{equation}
    Since $\theta_k^+>\theta_{k-1}$, this implies
    $$
    \|w(t)\|_{L^\infty}\le Ct^{-\theta_k^+/2}
    $$
    on $I_k$.

    Let $v=\mathbf P_{>2\lambda_{k-1}}\rho^2$ and
    $$
    t_k=\lambda_k^{-2\theta_k/\theta_k^+}>2\lambda_k^{-2}.
    $$
    Since $\mathbf P_{>2\lambda_{k-1}}$ is bounded on $L^\infty$ and $\rho^2$ satisfies the maximum principle,
    \begin{equation}\label{i3}
    \|v(t)\|_{L^\infty}\le C\|\rho^2(t)\|_{L^\infty}\le C\|\rho_{\mathrm{in}}\|_{L^\infty}\le C\lambda_k^{\theta_k}\le Ct^{-\theta_k^+/2}
    \end{equation}
    for $\lambda_k^{-2}\le t\le t_k$. On $[2\lambda_k^{-2},\lambda_{k-1}^{-2}]$, the function $v$ solves the heat equation because $b_k=0$. Hence, for $t\ge t_k$,
    $$
    \|v(t)\|_{L^\infty}\le\sum_{\xi}|\mathcal Fv(2\lambda_k^{-2})(\xi)|e^{- |\xi|^2(t-2\lambda_k^{-2})}.
    $$
    By \eqref{i3} and $t-2\lambda_k^{-2}\ge\frac12 t$ when $t\ge t_k$, we obtain
    $$
    \|v(t)\|_{L^\infty}\le\|v(2\lambda_k^{-2})\|_{L^\infty}\sum_{\xi}e^{- |\xi|^2t/2}\le\lambda_k^{\theta_k}\sum_{\xi}e^{- |\xi|^2t/2}.
    $$
    Because $\mathrm{supp}(\mathcal F\rho^2)\subset E_{k-1}$ and the cutoff separates the $z=0$ sector from the $z\ne0$ sector of $E_{k-1}$, the support of $v$ is contained in
    $$
    \bigcup_{z\in\mathbb Z^d\setminus\{0\}}z\lambda_k+\Omega_{k-1}.
    $$
    Therefore,
    $$
    \begin{aligned}
        \|v(t)\|_{L^\infty}
        &\le\lambda_k^{\theta_k}\sum_{z\in\mathbb Z^d\setminus\{0\}}\sum_{\xi\in\Omega_{k-1}}e^{-|\lambda_kz+\xi|^2t/2}\\
        &\le\lambda_k^{\theta_k}\sum_{z\in\mathbb Z^d\setminus\{0\}}2^{k-1}e^{-\lambda_k^2|z|^2t/8}\\
        &\le C(\frac {8N_k}e)^{N_k}\lambda_k^{\theta_k}2^{k-1}(\lambda_k^2t)^{-N_k}\sum_{z\in\mathbb Z^d\setminus\{0\}}|z|^{-2N_k},
    \end{aligned}
    $$
    where we used
    $$
    |\lambda_kz+\xi|\ge\frac12\lambda_k|z|
    $$
    because $\sum_{j=1}^{k-1}\lambda_j<\frac{\lambda_k}2$.

    By \eqref{st6},
    $$
    \|v(t)\|_{L^\infty}\le Ct^{-\theta_{k-1}/2}\le Ct^{-\theta_k^+/2}
    $$
    for $t_k\le t\le\lambda_{k-1}^{-2}$. At the final time $t=\lambda_{k-1}^{-2}$ we have $t-2\lambda_k^{-2}\ge\frac12\lambda_{k-1}^{-2}$, so the same summation argument gives
    \begin{equation}\label{bi4}
        \|v(\lambda_{k-1}^{-2})\|_{L^\infty}\le C(\frac{2N_k}e)^{N_k}\lambda_k^{\theta_k}2^k(\lambda_k^2\lambda_{k-1}^{-2})^{-N_k}\le\frac13\lambda_{k-1}^{\theta_{k-1}}
    \end{equation}
    by the parameter choice \eqref{st7}.

    Combining the bounds \eqref{bi1}\eqref{bi2}\eqref{i3}\eqref{bi4}, we conclude
    $$
    \|\rho(t)\|_{L^\infty}\le Ct^{-\theta_k^+/2}
    \qquad\text{and}\qquad
    \|\rho(\lambda_{k-1}^{-2})\|_{L^\infty}\le\lambda_{k-1}^{\theta_{k-1}}.
    $$
   
\end{proof}

The lemma above can now be applied to the recurrent part of the approximate correctors after splitting off the tail.

\begin{proposition}\label{cor.decay}
    For any $n$, we have $|\rho^n(t)|\le Ct^{-\theta_k^+/2}$ whenever $\lambda_k^{-2}\le t\le\lambda_{k-1}^{-2}$. Moreover,
    $$
    supp(\mathcal F\rho^n(\lambda_k^{-2}))\subset\Omega_k+\lambda_{k+1}\mathbb Z^d,
    $$
    for any $k$.
\end{proposition}
\begin{proof}
    Since $\rho^n(\lambda_n^{-2})=0$, this follows by repeatedly using Lemma \ref{boundary-iteration}.
\end{proof}

\begin{corollary}
    $\|\rho^n\|_{L^p(0,\infty;L^\infty(\tor))}\le C_p$ for all $1\le p<\infty$ and $n$.
\end{corollary}

\begin{proof}
    The estimate for $t\le 1$ is similar to the proof of Corollary \ref{principal-lp}. For $t\ge1$, we observe that $\rho^n$ solves the heat equation with $\li{\rho^n(1)}$ bounded and Fourier support in $\lambda_1\mathbb Z^d\setminus\{0\}$. Then
    $$
    \begin{aligned}
        \li{\rho^n(t)}&\le\sum_{z\in\lambda_1\mathbb Z^d\setminus\{0\}}|\hat\rho(z)(t)|\le\sum_{z\in\lambda_1\mathbb Z^d\setminus\{0\}}|\hat\rho^n(z)(2\lambda_1^{-2})|e^{-|z|^2(t-2\lambda_1^{-2})}\\
        &\le \li{\rho^n(2\lambda_1^{-2})}\sum_{z\in\lambda_1\mathbb Z^d\setminus\{0\}}e^{-|z|^2(t-2\lambda_1^{-2})},
    \end{aligned}
    $$
    which is finite and decays exponentially for $t\ge1$.    
\end{proof}

\section{Compactness and nontriviality}
We now pass to the limit in the principal part and in the approximate correctors.

\begin{proposition}\label{weak.sol.p}
$\varrho$ solves
\[
\partial_t\varrho+b\cdot\nabla\varrho-\Delta\varrho=f
\]
with zero initial data in the sense of distributions.
\end{proposition}

\begin{proof}
Let $\varphi\in C_c^\infty([0,\infty)\times\tor)$. We prove
\[
\int_0^\infty\!\!\int_{\tor}
\varrho(\partial_t\varphi+b\cdot\nabla\varphi+\Delta\varphi)\,dx\,dt
=
-\langle f,\varphi\rangle.
\]
Let $\varrho_N:=\sum_{k=1}^N\rho_k$ and, for $1\le j\le N$,
\[
A_{j,N}:=\rho_{j-1}+(\rho_j-\bar\rho_j)+\sum_{j<k\le N}\rho_k .
\]
For $j=1$, we use $\rho_0=0$ and $b_1\cdot\nabla\bar\rho_1=0$.
Then
\[
\partial_t\varrho_N+b\cdot\nabla\varrho_N-\Delta\varrho_N
=
\sum_{j=1}^N b_j\cdot\nabla A_{j,N}
+b_{N+1}\cdot\nabla(\rho_N-\bar\rho_{N+1}).
\]
Integrating by parts, we obtain
\[
\begin{aligned}
&\int_0^\infty\!\!\int_{\tor}
\varrho_N(\partial_t\varphi+b\cdot\nabla\varphi+\Delta\varphi)\,dx\,dt\\
&=
\sum_{j=1}^N\int_0^\infty\!\!\int_{\tor}
A_{j,N}b_j\cdot\nabla\varphi\,dx\,dt
+
\int_0^\infty\!\!\int_{\tor}
(\rho_N-\bar\rho_{N+1})b_{N+1}\cdot\nabla\varphi\,dx\,dt .
\end{aligned}
\]

The integration of $\varrho_N(\partial_t\varphi+\Delta\varphi)$ converges because
\[
\int_0^\infty\|\rho_k(t)\|_{L^\infty}\,dt \le
C\lambda_k^{\alpha_k-2}.
\]
For the transport term, Proposition \ref{key} and maximal principle give
\[
\|\rho_k\|_{L_t^\infty L_x^\infty} \le
C\lambda_k^{\alpha_k}.
\]
Hence, by interpolation,
\[
\|\rho_k\|_{L_t^3L_x^\infty}
\le
C\lambda_k^{\alpha_k-2/3}.
\]
This is summable in $k$, so
\[
\varrho_N\to \varrho
\qquad\text{strongly in }L_t^3L_x^\infty.
\]
Since $b\in L_t^{3/2}L_x^\infty$, the transport term converges by
Hölder's inequality.

The estimates from Proposition~\ref{key} give
\[
\int_{J_j}\|A_{j,N}b_j\|_{L^\infty}\,dt
\le C\Bigl( \lambda_{j-1}^{2\alpha_{j-1}-1}\lambda_j^{-\alpha_j}
+\lambda_j^{-1}\lambda_{j-1}^{\alpha_{j-1}} \Bigr),
\]
which is summable in $j$, and the top-scale term tends to zero. Passing to
the limit,
\[
\begin{aligned}
    \int_0^\infty\!\!\int_{\tor}
\varrho(\partial_t\varphi+b\cdot\nabla\varphi+\Delta\varphi)\,dx\,dt
&=\sum_{j\ge1}\int_0^\infty\!\!\int_{\tor}A_jb_j\cdot\nabla\varphi\,dx\,dt\\
&= -\langle f,\varphi\rangle,  
\end{aligned}
\]
where
\[
A_j:=\rho_{j-1}+(\rho_j-\bar\rho_j)+\sum_{k>j}\rho_k .
\]
\end{proof}

\begin{proposition}\label{corrector-compactness}
There exist $n_\ell\to\infty$ and
\[
\rho_c\in\bigcap_{1\le p<\infty}L^p(0,\infty;L^\infty(\tor))
\]
such that, for every $0<\tau<T<\infty$,
\[
\rho^{n_\ell}\stackrel{*}{\rightharpoonup}\rho_c
\quad\text{in }L^\infty([\tau,T]\times\tor).
\]
Moreover,
\[
\partial_t\rho_c+b\cdot\nabla\rho_c-\Delta\rho_c=f
\]
in $\mathcal D'([0,\infty)\times\tor)$.
\end{proposition}

\begin{proof}
By Proposition~\ref{cor.decay}, we have uniform $L_{t,x}^\infty$ bound for $\rho^n$ on any time interval $[\tau,T]$ away from zero, as well as uniform global $L^pL^\infty$ bound. Therefore, Banach--Alaoglu and a diagonal argument give the local weak-$*$ limit, and the same envelope gives the stated integrability of
$\rho_c\in L^p(0,\infty;L^\infty(\tor))$.

For every $n$,
\[
\left\langle
\rho^n,-\partial_t\varphi-b\cdot\nabla\varphi-\Delta\varphi
\right\rangle
= \left\langle\sum_{k\le n}f_k,\varphi\right\rangle .
\]
On $[\delta,T]$, pass to the limit by local weak-$*$ convergence. On
$(0,\delta)$, choose $1<q<2$ and let $q'$ be conjugate to $q$:
\[
\int_0^\delta\|b(t)\rho^n(t)\|_{L^\infty}\,dt
\le \|b\|_{L^q(0,\delta;L^\infty)}
\sup_n\|\rho^n\|_{L^{q'}(0,T;L^\infty)} \to 0.
\]
The same estimate holds for $\rho_c$, and the non-transport terms are handled
by the uniform $L^p_tL^\infty_x$ bounds. Finally,
$\sum_{k\le n_\ell}f_k\to f$ in distributions by the summability above.
\end{proof}

Now we show smoothness of $\varrho-\rho_c$ for positive time.

\begin{corollary}\label{smooth-remainder}
$\varrho-\rho_c$ is a weak solution to \eqref{adv} that belongs to $C^\infty((0,\infty)\times\tor)$.
\end{corollary}

\begin{proof}

By Proposition \ref{weak.sol.p} and Proposition \ref{corrector-compactness}, $\varrho-\rho_c$ solves
$$
\partial_t\rho+b\cdot\nabla\rho-\Delta\rho=0
$$
in $(0,\infty)\times\tor$.

Fix $0<\tau<T<\infty$. By Proposition \ref{principal-bounds} and Proposition \ref{corrector-compactness},
$$
\varrho-\rho_c\in L^\infty([\tau,T]\times\tor).
$$
Since the drift $b$ is smooth on $[\tau,T]\times\tor$, standard interior regularity for uniformly parabolic equations with smooth coefficients therefore yields
$$
\varrho-\rho_c\in C^\infty((\tau,T)\times\tor).
$$
Since $\tau$ and $T$ were arbitrary, the claim follows.

\end{proof}

\begin{corollary}\label{nontrivial}
    $\varrho-\rho_c$ is non-trivial.
\end{corollary}

\begin{proof}
    By Corollary \ref{smooth-remainder}, the $L^\infty$ norm of $\rho_c$ at each time makes sense. By Proposition \ref{cor.decay}, $|\rho^n(t)|\le Ct^{-\theta_k^+/2}$ for $\lambda_k^{-2}\le t\le\lambda_{k-1}^{-2}$. By Proposition \ref{corrector-compactness}, this uniform bound also holds for $\rho_c$. Then by Proposition \ref{principal-bounds}, we have
    $$
    \frac{\li{\varrho(\lambda_k^{-2})}}{\li{\rho_c(\lambda_k^{-2})}}\ge C\lambda_k^{\alpha_k-\theta_k^+}\to\infty
    $$
    since $\lambda_k\sim\lambda_1^{ck^3+O(k^2)}$. Therefore, $\varrho-\rho_c$ is non-trivial.

\end{proof}

\section{A weak-$L^2$ threshold regime}

This section uses a new cascade, independent of the construction in
Theorem~\ref{main}. All parameter-dependent objects, including
$\phi_k,\bar\rho_k,b_k,\rho_k,f_k,f,\rho^n,$ and $\rho_c$, are rebuilt
from the borderline choice of parameters below.

\begin{proposition}\label{weak-l2-choice}
There exist a divergence-free drift
\[
b \in C^\infty((0,\infty)\times\tor) \cap L_t^{2,\infty}L_x^\infty
\cap L_{t,x}^{2,\infty}
\]
and a nonzero function
\[
\rho\in C^\infty((0,\infty)\times\tor)
\cap\bigcap_{0<r\le\infty}L_t^{2,r}L_x^\infty
\]
such that $b\rho\in L_t^1L_x^\infty$ and $\rho$ solves \eqref{adv}
with zero initial datum in the sense of \eqref{weak-form-intro}.
\end{proposition}

\begin{proof}
We use the same building blocks and truncated corrector scheme as before.
We only need to choose the parameters to fit the passive scalar and vector field in the regularity class claimed above and verify the parameter conditions with updated estimates in Proposition \ref{parameter}.

\textbf{Choice of Parameters.}
We fix $\kappa>0$ sufficiently large and let
\[
s_k=1+\frac2{k+2},\qquad
\lambda_{k+1}=\lambda_k[\lambda_k^{s_k-1}],\qquad
\lambda_0=\alpha_0=1,\qquad L_k:=\log\lambda_k,
\]
where $\lambda_1\in\mathbb N$, and
\[
\delta_k:=\frac{\kappa(k+2)}{L_k},\qquad
\alpha_k:=1-\delta_k,\qquad
\theta_k:=\alpha_k-\delta_k,\qquad
\theta_k^+:=\frac{\alpha_k+\theta_k}{2}.
\]
For a constant $C_0>0$ determined by $\lambda_1$,
\[
L_k=C_0k^2+O(k), \qquad
\delta_k\sim\frac{\kappa}{C_0k}.
\]
Choose $\lambda_1$ so that $C_0/\kappa$ is sufficiently large and fixed.
Then the usual lacunarity conditions hold and
\[
0<\theta_{k-1}<\theta_k<\theta_k^+<\alpha_k<1, \qquad
\frac{\theta_k}{\theta_k^+}>\frac{L_{k-1}}{L_k}, \qquad
2s_k+s_{k-1}^{-1}>3.
\]

On $J_k$,
\[
\|b_k\|_{L^\infty} \lesssim
\lambda_{k-1}^{-\delta_{k-1}}\lambda_k^{1+\delta_k}, \qquad
|J_k|\,\|b_k\|_{L^\infty}^2\lesssim e^{2\kappa}.
\]
Hence,
\[
\|b(t)\|_{L^\infty}\lesssim e^\kappa t^{-1/2}, \qquad
b\in L_t^{2,\infty}L_x^\infty \cap L_{t,x}^{2,\infty}.
\]

\textbf{Parameter Conditions.}
Take $m_k=k+2$.  Since $0<\alpha_k<1$, the exponent gaps in
\eqref{st1}--\eqref{st3} are bounded below uniformly; the factors
$(Cm_k)^{m_k}$ are therefore absorbed by $L_k\sim k^2$.

Since $\delta_jL_j=\kappa(j+2)$,
\[
\frac{\lambda_{k-1}^{\alpha_{k-1}+1}\lambda_k^{-1}} {\lambda_{k-1}^{\theta_{k-1}}} = e^{\kappa(k+1)-(L_k-L_{k-1})}, \qquad
\frac{\lambda_{k-1}^{2\alpha_{k-1}}\lambda_k^{-\alpha_k}} {\lambda_{k-1}^{\theta_{k-1}}} = e^{\kappa(k+2)-(L_k-L_{k-1})},
\]
and
\[
\lambda_{k-1}^{\alpha_{k-1}-\theta_{k-1}} \lambda_k^{-\alpha_k+\theta_k}
=e^{-\kappa}.
\]

Thus \eqref{st4} follows by taking $C_0/\kappa$ sufficiently large,
and \eqref{st5} follows by taking $\kappa$ sufficiently large.

For \eqref{st6}--\eqref{st7}, take
\[
N_k=\lceil M(k+2)\rceil,
\]
where $M$ is sufficiently large, depending only on $d$ and the fixed
ratio $C_0/\kappa$. The corresponding exponent gains are
\[
A_k := 2N_k\left(1-\frac{\theta_k}{\theta_k^+}\right) + \theta_k\left(\frac{\theta_{k-1}}{\theta_k^+}-1\right)
\]
and
\[
B_k := 2N_k\left(1-\frac{L_{k-1}}{L_k}\right) + \frac{L_{k-1}}{L_k}\theta_{k-1}-\theta_k.
\]
Since
\[
1-\frac{\theta_k}{\theta_k^+} \sim k^{-1}, \qquad \theta_k^+-\theta_{k-1} \sim k^{-1}, \qquad 1-\frac{L_{k-1}}{L_k}\sim k^{-1},
\]
we have $A_k$, $B_k \ge c>0$ for $M$ sufficiently large. Finally, the remaining factors are $e^{O(k\log k)}$, whole $L_k \sim k^2$, so \eqref{st6}--\eqref{st7} follow.

\textbf{Principal Part.}
Proposition~\ref{key} and the same Duhamel calculation give
\[
\|\rho_k\|_{L_t^\infty L_x^\infty}
\le C\lambda_k^{\alpha_k}, \qquad
\|\rho_k\|_{L_t^1L_x^\infty}
\le C\lambda_k^{\alpha_k-2}.
\]
Hence, for every $0<r\le\infty$,
\[
\|\rho_k\|_{L_t^{2,r}L_x^\infty} \lesssim_r
\lambda_k^{\alpha_k-1} =
e^{-\kappa(k+2)}.
\]
Thus $\varrho:=\sum_k\rho_k$
converges in $L_t^{2,r}L_x^\infty$ for every $0<r\le\infty$.

Now let $\tau_k:=\lambda_k^{-2}$. The lower-bound proof of Proposition~\ref{principal-bounds} is unchanged except that
\[
\frac{\lambda_k\lambda_{k+1}^{\alpha_{k+1}-1}}
{\lambda_k^{\alpha_k}} =e^{-\kappa}, \qquad
\|\bar\rho_{k+1}(\tau_k)\|_{L^\infty} \lesssim
\lambda_{k+1}^{\alpha_{k+1}}
e^{-(\lambda_{k+1}/\lambda_k)^2} =o(\lambda_k^{\alpha_k}),
\]
and
\[
\sum_{j\ge k+2}\lambda_{j-1}^{\alpha_j-1}
\lesssim e^{-c\kappa k}.
\]
Therefore, after increasing $\kappa$,
\[
\|\varrho(\tau_k)\|_{L^\infty} \gtrsim \lambda_k^{\alpha_k}.
\]

\textbf{Corrector and Passage to the Limit.}
The iteration lemma gives, uniformly in $n$,
\[
\|\rho^n(t)\|_{L^\infty} \le
Ct^{-\theta_k^+/2} \le
Ce^{-c\kappa k}t^{-1/2}, \qquad t\in I_k.
\]
Together with the uniform positive-time decay, this yields a common upper bound
\[
G(t):=
C\sum_{k\ge1}
e^{-c\kappa k}t^{-1/2}\mathbf 1_{I_k}(t) +
Ce^{-ct}\mathbf 1_{[1,\infty)}(t) \in\bigcap_{0<r\le\infty}L^{2,r}(0,\infty),
\]
where we used
\[
\log\frac{\lambda_k}{\lambda_{k-1}}=O(k).
\]
Hence a subsequence converges locally weak-$*$ to a corrector satisfying
\[
\rho_c\in \bigcap_{0<r\le\infty}L_t^{2,r}L_x^\infty, \qquad
|\rho_c|\le G,
\]
and
\[
\|\rho_c(t)\|_{L^\infty} \lesssim t^{-\theta_k^+/2}, \qquad t\in I_k.
\]

The shell-wise estimate defining $f$ remains summable, since
\[
\lambda_{j-1}^{2\alpha_{j-1}-1}\lambda_j^{-\alpha_j} + \lambda_j^{-1}\lambda_{j-1}^{\alpha_{j-1}}
\lesssim e^{-cj}.
\]
Thus the finite-sum arguments in Proposition~\ref{weak.sol.p} and
\ref{corrector-compactness} carry over with Lorentz H\"older. The only new bounds are
\[
\left\| b\sum_{j>N}\rho_j \right\|_{L_t^1L_x^\infty}
\lesssim \|b\|_{L_t^{2,\infty}L_x^\infty}
\left\| \sum_{j>N}\rho_j \right\|_{L_t^{2,1}L_x^\infty} \to 0,
\]
\[
\begin{aligned}
&\int_{J_{N+1}} \|b_{N+1}(\rho_N-\bar\rho_{N+1})\|_{L^\infty}\,dt\\
&\qquad\lesssim e^{-(L_{N+1}-L_N)-\kappa(N+1)} + e^{-\kappa(N+2)} \to 0,
\end{aligned}
\]
and
\[
\sup_n\int_0^\delta \|b(t)\rho^n(t)\|_{L^\infty}\,dt \lesssim \|b\|_{L_t^{2,\infty}L_x^\infty}
\|G\mathbf 1_{(0,\delta)}\|_{L_t^{2,1}} \to 0.
\]
The remaining terms are controlled by $G\in L^1_{\mathrm{loc}}$. Thus $\varrho$ and $\rho_c$ solve the same forced equation with zero initial data.

Let
\[
\rho:=\varrho-\rho_c.
\]
Since $b$ is smooth for positive time, we have $\rho\in C^\infty((0,\infty)\times\tor)$. The estimate for the corrector implies that $\rho$ is nonzero. Finally,
\[
\|b\rho\|_{L_t^1L_x^\infty}
\lesssim \|b\|_{L_t^{2,\infty}L_x^\infty}
\|\rho\|_{L_t^{2,1}L_x^\infty}
<\infty.
\]

\end{proof}

\section{Bounded-energy inverse-mixing construction}

For $A>0$, let
\[
F_A(s)=s^2\exp\!\left[-A\bigl(\log(e+s)\bigr)^{2/3}\right],
\qquad s\ge0.
\]

Our main construction is the Fourier space produces a diffusion-assisted inverse cascade of energy that blows up as $t \to 0$. Here we pursue a different scenario where only
a finite amount of energy starts cascading from the infinite frequencies at the initial time. This can only be achieved via almost pure inverse energy cascade with perturbative diffusion. We use the inviscid perfect-mixing construction of
Alberti, Crippa, and Mazzucato and run each inverse energy transfer step on a time
scale much shorter than diffusion.  This yields a bounded solution in the global
energy space.

\begin{theorem}
\label{thm:bounded-energy}
For every $d\ge2$, there exist a divergence-free drift
\[
b\in C^\infty((0,\infty)\times\tor)
\]
and a nonzero solution
\[
\rho\in C^\infty((0,\infty)\times\tor) \cap L^\infty((0,\infty)\times\tor)
\cap L^\infty(0,\infty;L^2(\tor)) \cap L^2(0,\infty;H^1(\tor))
\]
of \eqref{adv} with zero initial datum. Moreover,
\begin{equation}\label{be-orlicz}
\int_0^\infty G_A\!\left(\|b(t)\|_{L^\infty}\right)\,dt<\infty
\qquad \forall A>0,
\end{equation}
while
\begin{equation}\label{be-not-weak-L2}
b\notin L^{2,\infty}(0,T;L^\infty(\tor)) \qquad \forall T>0.
\end{equation}
The scalar has zero mean,
\[
\rho\in C_w([0,\infty); L^2(\tor)),  \qquad \rho(0)=0, \qquad \rho\in\bigcap_{1\le p<\infty}L^p(0,\infty;L^\infty(\tor)),
\]
and there is $E_* > 0$ such that
\begin{equation}\label{injected_energy}
\lim_{t \to 0}\|\rho(t)\|_{L^2}^2=E_*.
\end{equation}
For every $0<t_0<t<\infty$,
\begin{equation}\label{be-energy_equality_positive}
\frac12 \|\rho(t)\|_{L^2}^2  + \int_{t_0}^t\|\nabla\rho(\tau)\|_{L^2}^2\,d \tau= \frac12\|\rho(t_0)\\\|_{L^2}^2,
\end{equation}
and hence
\begin{equation}\label{be-global-energy}
\frac12 \|\rho(t)\|_{L^2}^2 +\int_0^t\|\nabla\rho(\tau)\|_{L^2}^2\,d\tau =\frac12E_*.
\end{equation}
\end{theorem}
As in the previous constructions, the solution satisfies the energy equality starting from any $t_0>0$, but not the initial time $t_0=0$. The amount of energy that cascades from the infinite frequency is $E_*$, which produces the initial energy defect.

{\bf Perfect-mixing building blocks}

We will use the smooth quasi-self-similar mixer of
\cite[Sections~6 and~8]{ACM}.  The following stagewise formulation is obtained
by applying a flattening reparametrization near $t=0$ and $t=1$, and reversing time.
\begin{proposition}\label{be-mixer}
There are profiles $\vartheta_k\in C^\infty(\mathbb T^d)$, $k\ge0$,
and smooth pairs $(U_k,\Theta_k)$ on
$[0,1]\times\mathbb T^d$, $k\ge1$, such that, with $\ell_k=5^{-k}$,
\begin{align}
&\dv U_k=0, \qquad \partial_s \Theta_k+U_k\cdot \nabla\Theta_k=0, \qquad
\Theta_k(0)=\vartheta_k, \quad
\Theta_k(1)=\vartheta_{k-1};  \label{be-block-transport}\\
&\int_{\mathbb T^d}\vartheta_k\, dx=0, \qquad \|\vartheta_k\|_{L^2}=1, \qquad
\|\vartheta_k\|_{L^\infty}\le C; \label{be-block-norms}\\
&\|\vartheta_k\|_{\dot H^{-1}}\le C\ell_k, \qquad \int_0^1 \|\nabla \Theta_k(t)\|_{L^2}^2\,dt\le C\ell_k^{-2}; \label{be-block-cost}\\
&\sup_{0<s<1}\|U_k(s)\|_{L^\infty}\le C\ell_{k-1}.
\label{be-block-velocity}
\end{align}
All time derivatives of $U_k$ vanish at $t=0,1$.  The blocks may also be chosen so that
\begin{equation}\label{be-block-lower}
\left|\left\{t\in(0,1):
\|U_k(t)\|_{L^\infty}\ge c\ell_{k-1}\right\}\right|\ge c.
\end{equation}
\end{proposition}

{\bf Compressed inverse energy cascade}

Take $\varepsilon_0>0$ and let
\begin{equation}\label{be-parameters}
\varepsilon_k := \frac{\varepsilon_0}{(k+1)^2}, \qquad \delta_k :=\varepsilon_k^2 \ell_k^2, \qquad \tau_k := \sum_{j=k+1}^\infty \delta_j, \qquad J_k:=[\tau_k,\tau_{k-1}].
\end{equation}
Define the drift 
\begin{equation}\label{be-drift}
b(t,x):= \left\{
\begin{split}
\frac{1}{\delta_k} U_k\!\left(\frac{t-\tau_k}{\delta_k},x\right),  \qquad  &t\in J_k,\\
0, \qquad & t \ge \tau_0.
\end{split}
\right.
\end{equation}
Note that
\begin{equation}\label{be-Bk}
\|b(t)\|_{L^\infty}\le
B_k:=C\delta_k^{-1}\ell_{k-1}
\lesssim\varepsilon_k^{-2}\ell_k^{-1},
\end{equation}
for $t\in J_k$. Therefore, for $1\le q<2$,
\[
\int_{J_k}\|b(t)\|_{L^\infty}^q\,dt \lesssim \varepsilon_k^{2-2q}\ell_k^{2-q}.
\]
Moreover,
\[
\delta_kB_k^2\lesssim\varepsilon_k^{-2}, \qquad \log(e+B_k) \gtrsim k,
\]
and
\[
\int_{J_k}G_A(\|b(t)\|_{L^\infty})\,dt \lesssim_A  \delta_k\varepsilon_k^{-2}e^{-cAk^{3/4}}.
\]
Hence \eqref{be-orlicz} holds. By \eqref{be-block-lower}, on a subset of
$J_k$ of measure comparable to $\delta_k$,
\[
\|b(t)\|_{L^\infty}\gtrsim\delta_k^{-1}\ell_{k-1},
\]
which implies \eqref{be-not-weak-L2}.

\begin{lemma}\label{perturbative-diffusion}
Let $\Theta$ be the solution of 
$\partial_t\Theta+U\cdot\nabla\Theta=0$ on $[0,1]$, with $\dv U=0$.
Let
\[
\widetilde b(t,x):=\delta^{-1}U\!\left(\frac{t-a}{\delta},x\right), 
\]
and let $w$ be the solution of
$\partial_tw+\widetilde b\cdot\nabla w-\Delta w=0$.  Then
\begin{equation}\label{eq:perturbative-diffusion-bound}
\|w(a+\delta)- \Theta(1)\|_{L^2}
\le \|w(a)-\Theta(0)\|_{L^2} + \delta^{1/2} \left(\int_0^1\|\nabla\Theta(t)\|_{L^2}^2\, dt \right)^{1/2}.
\end{equation}
\end{lemma}
\begin{proof}
Consider $\widetilde\Theta(t)=\Theta((t-a)/\delta)$ and $z=w-\widetilde\Theta$, which satisfies
\[
\partial_tz+\widetilde b\cdot\nabla z-\Delta z = \Delta\widetilde\Theta.
\]
Then the energy estimate implies \eqref{eq:perturbative-diffusion-bound}.
\end{proof}

By \eqref{be-parameters}, \eqref{be-block-cost}, and
Lemma~\ref{perturbative-diffusion},
\begin{equation}\label{eq:one_block-diffusion-bound}
\|w(\tau_{k-1})-\vartheta_{k-1}\|_{L^2} \le \|w(\tau_k)-\vartheta_k\|_{L^2}+C\varepsilon_k.
\end{equation}
Since $\sqrt{\delta_k}=\varepsilon_k\ell_k\ll\ell_k$,
the diffusion acts only on a subparabolic time scale.

{\bf Passage to the limit}

\begin{proof}[Proof of Theorem~\ref{thm:bounded-energy}]
For $n \ge 1$, let $\rho^{(n)}$ be the solution of \eqref{adv} on
$[\tau_n,\infty)$ with
\[
\rho^{(n)}(\tau_n)=\vartheta_n,
\]
and extend it by zero to $(0,\tau_n)$. Iterating \eqref{eq:one_block-diffusion-bound}, we obtain
\begin{equation} \label{eq:final_energy}
\|\rho^{(n)}(\tau_0)- \vartheta_0 \|_{L^2} \le C \sum_{k=1}^n \varepsilon_k.
\end{equation}
By the maximum principle and energy equality, we get, uniformly in $n$,
\begin{equation} \label{eq:bounds_on_rho^n}
\|\rho^{(n)}\|_{L^\infty_{t,x}} \le C, \qquad \sup_{t\ge0} \|\rho^{(n)}(t)\|_{L^2} \le 1, \qquad
\int_0^\infty \|\nabla\rho^{(n)}\|_{L^2}^2\,dt \le \frac12.
\end{equation}
After passing to a subsequence,
\[
\rho^{(n)}\stackrel{*}{\rightharpoonup}\rho \quad \text{in }L^\infty_{t,x}\cap L^\infty_tL^2_x,
\qquad \rho^{(n)}\rightharpoonup\rho \quad\text{in }L^2_tH^1_x.
\]
Since $\rho^{(n)}$ is a classical solution,
\[
\int_0^\infty\!\!\int_{\tor}\rho^{(n)} \bigl(\partial_t\varphi+b\cdot\nabla\varphi+\Delta\varphi\bigr) \,dx\,dt = -\int_{\tor}\vartheta_n\varphi(\tau_n)\,dx,
\]
for every test function $\varphi$. We can pass to the limit on the left-hand side because $\partial_t\varphi + \Delta\varphi + b\cdot\nabla\varphi\in L^1_{x,t}$. Also, since $\vartheta_n$ has zero mean, by \eqref{be-block-cost},
\[
\left| \int_{\tor}\vartheta_n\varphi(\tau_n)\,dx \right| \le \|\vartheta_n\|_{\dot H^{-1}}
\|\nabla\varphi(\tau_n)\|_{L^2} \lesssim \ell_n\|\nabla\varphi\|_{L_t^\infty L_x^2} \to 0.
\]
Hence $\rho$ satisfies \eqref{weak-form-intro}, i.e., it is a distributional solution with zero initial datum.

Passing to another subsequence, $\rho^{(n)}(\tau_0)\rightharpoonup \theta_0$ in $L^2$ and weakly-$*$ in $L^\infty$. Taking $\varepsilon_0$ small enough, by \eqref{eq:final_energy} and \eqref{eq:bounds_on_rho^n}, 
\begin{equation}\label{eq:final_energy_comparison}
\|\theta_0-\vartheta_0\|_{L^2}\le\frac14,
\end{equation}
and hence $\|\theta_0\|_{L^2} \ge 3/4$. Since $b=0$ for $t\ge\tau_0$,
\[
\rho(t)=e^{(t-\tau_0)\Delta}\theta_0, \qquad t\ge\tau_0,
\] 
and consequently it is nonzero.

Since the drift is smooth for positive time, $\rho \in C^\infty((0,\infty)\times\tor)$ by parabolic regularity and the energy equality \eqref{be-energy_equality_positive} holds. Thanks to $b\rho\in L^1((0,T)\times\tor)$, we also have $\rho\in C_w([0,\infty);L^2)$. Hence, since $\rho(t)$ attains zero initial datum in the sense of distributions, we have $\rho(t) \rightharpoonup 0$ in $L^2$ as $t \to 0$.

The energy function $t\mapsto\|\rho(t)\|_{L^2}$ is nonincreasing, so the limit
$E_*$ in \eqref{injected_energy} exists.  By \eqref{eq:final_energy_comparison},
\[
E_*\ge\|\rho(\tau_0)\|_{L^2}^2 = \|\theta_0\|_{L^2}^2\ge\frac9{16}.
\]
Letting $t_0 \to 0$ in \eqref{be-energy_equality_positive}, we obtain
\eqref{be-global-energy}. Finally, since $\rho \in L^\infty((\tau,\infty)\times \tor))$ for every $\tau >0$, we have $\rho\in L^p(0,\infty;L^\infty)$ for every finite $p$ by parabolic estimates.

\end{proof}

\section{Appendix}
Here we prove uniqueness in the weak-$L^2_t$ endpoint assuming an energy inequality.

\begin{proposition}
Assume $b$ is divergence free, $b \in L^{2,\infty}(0,T;L^\infty(\mathbb T^d))$, and
\[
w\in L^\infty(0,T;L^2(\mathbb T^d)) \cap L^2_{\rm loc}((0,T);H^1(\tor)),
\]
is a distributional solution of
\[
\partial_t w+\operatorname{div}(bw)-\Delta w=0
\]
with $w|_{t=0}=0$. Assume moreover that $w$ satisfies the energy inequality
\[
\frac12\|w(t)\|_{L^2}^2 + \int_{t_0}^t\|\nabla w(\tau)\|_{L^2}^2\,d\tau \le \frac12\|w(t_0)\|_{L^2}^2,
\]
for every $t \in (t_0,T)$, almost all $t_0 \in (0,T)$. Then $w\equiv0$.
\end{proposition}
\begin{proof}
Let
\[ 
K:=\|b\|_{L^{2,\infty}_tL^{\infty}_x}, \qquad M:=\|w\|_{L_t^\infty L_x^2}.
\]
Choose a sequence of times $t_j \to 0^+$ starting from which the energy inequality holds. Then, for every $t<T$,
\[
\int_{t_j}^{t}\|\nabla w(s)\|_{L^2}^2\,ds \le\frac12\|w(t_j)\|_{L^2}^2
\le\frac12M^2.
\]
Letting first $j\to\infty$ and then $t \to T^-$, we obtain
\begin{equation}\label{eq:weak-L^2_global_dissipation}
\int_0^T\|\nabla w(s)\|_{L^2}^2\,ds
\le\frac12M^2.
\end{equation}

The weak-$L^2$ estimate implies
\[
\int_0^t  \|b(s)\|_{L^\infty}\,ds\le 2K\sqrt t, \qquad 0<t<T.
\]
We also have
\[
bw \in L^1(0,T;L^2(\mathbb T^d)).
\]

Let $P_{\le N}$ denote the Fourier projection onto frequencies $|\xi|\le N$.
From
\[
P_{\le N}w(t) = -\int_0^t e^{(t-s)\Delta}
P_{\le N}\operatorname{div}(bw)(s)\,ds,
\]
we obtain
\[
\|P_{\le N}w(t)\|_{L^2}  \le C_0 KMN\sqrt t.
\]

Suppose that $w\not\equiv 0$. Then $K,M>0$. Choose a time
$s\in(0,T)$, such that
\[
m:=\|w(s)\|_{L^2}>0.
\]
The energy inequality implies
\[
\|w(t)\|_{L^2}\ge m,
\qquad \quad \text{for a.a.} \quad t\in(0,s).
\]
For such sufficiently small $t$,
\[
\|P_{\le N(t)}w(t)\|_{L^2} \le \frac m2, \qquad N(t):=
\left\lfloor \frac{m}{2C_0KM\sqrt t} \right \rfloor,
\]
and consequently
\[
\|P_{>N(t)}w(t)\|_{L^2}  \ge \|w(t)\|_{L^2} -
\|P_{\le N(t)}w(t)\|_{L^2} \ge \frac{m}{2}.
\]
It follows that
\[
\begin{aligned}
\|\nabla w(t)\|_{L^2}^2
&\gtrsim N(t)^2\|P_{>N(t)}w(t)\|_{L^2}^2\\
&\gtrsim \frac{m^4}{K^2M^2}\frac1t.
\end{aligned}
\]
Therefore
\[
\int_0^{T}\|\nabla w(t)\|_{L^2}^2\,dt=\infty,
\]
contradicting \eqref{eq:weak-L^2_global_dissipation}. Hence $w\equiv0$.

\end{proof}

\bibliographystyle{alpha}
\bibliography{intro_advection-diffusion-refs}

\end{document}